\documentclass[arxiv,reqno,twoside,a4paper,12pt]{amsart}

\usepackage{amsmath, verbatim}
\usepackage{amssymb,amsfonts,mathrsfs,mathtools}
\usepackage[colorlinks=true,linkcolor=blue,urlcolor=blue,citecolor=blue]{hyperref}
\usepackage[driver=pdftex,heightrounded=true,centering]{geometry}
\usepackage{longtable, slashed}
\usepackage{tabularx}
\usepackage{multirow}
\usepackage{caption}
\usepackage{latexsym,enumitem}
\usepackage[all]{xy}
\usepackage[latin2]{inputenc}
\usepackage{t1enc, indentfirst}
\usepackage{graphicx, graphics}
\usepackage{pict2e, epic, amssymb }
\usepackage{tikz, pst-node}
\usepackage{tikz-cd, pgfplots}
\usetikzlibrary{arrows}
\usetikzlibrary{calc, patterns}
\pgfplotsset{compat=1.9}
\usepackage{epstopdf}

\usepackage[toc]{appendix}
\usepackage[mathbf]{euler}

\usepackage[UKenglish]{babel}
\usepackage{tikz}
\usepackage{tikz-3dplot}

\usepackage{diagmac2} 
\usepackage{array} 
\usetikzlibrary{positioning}
\usetikzlibrary{calc}
\usetikzlibrary{through}
\usepackage{subcaption}

\usepackage[scaled]{helvet} 
\usepackage{courier}

\theoremstyle{break}
\newtheorem{thm}{Theorem}[section]
\newtheorem{lem}[thm]{Lemma}%
\newtheorem{prop}[thm]{Proposition}
\newtheorem{cor}[thm]{Corollary}
\newtheorem{defn}[thm]{Definition}
\newtheorem{rmk}[thm]{Remark} 

\newtheorem{assump}[thm]{Assumptions}
\newtheorem{setting}[thm]{Setting}
\newtheorem{notation}[thm]{Notation}

\DeclareMathOperator{\trace}{trace}
\DeclareMathOperator{\sff}{II}
\DeclareMathOperator{\Id}{id}
\DeclareMathOperator{\ssff}{h}

\DeclareMathOperator{\smc}{H}
\DeclareMathOperator{\divergence}{div}

\DeclareMathOperator{\dvol}{dvol}
\DeclareMathOperator{\di}{d}

\DeclareMathOperator{\ric}{Ric}
\newcommand{\PBD}{\nabla^{F^*TN}}
\newcommand{\SPBD}{\nabla^{\mathcal{F}^*TN}}
\newcommand{\gbar}{\overline{g}}
\newcommand{\pr}[1]{\left(#1\right)}

\DeclareMathOperator{\vol}{Vol}
\DeclareMathOperator{\R}{\mathbb{R}}
\DeclareMathOperator{\N}{\mathbb{N}}
\DeclareMathOperator{\NT}{\widetilde{\nabla}}

\numberwithin{equation}{section}

\definecolor{qqwuqq}{rgb}{0,0,0}

\begin{document}

\title[Mean curvature flow of non-compact Cauchy hypersurfaces]{Prescribed mean curvature flow of non-compact space-like Cauchy hypersurfaces}

\author{Giuseppe Gentile}
\address{Universit\"{a}t Hannover, Germany}
\email{giuseppe.gentile@math.uni-hannover.de}

\author{Boris Vertman}
\address{Universit\"{a}t Oldenburg, Germany}
\email{boris.vertman@uni-oldenburg.de}


\maketitle
\begin{abstract}
In this paper we consider the prescribed mean curvature flow of a 
non-compact space-like Cauchy hypersurface of bounded geometry in a generalized Robertson-Walker space-time.
We prove that the flow preserves the space-likeness condition and exists for infinite time. We also prove convergence 
in the setting of manifolds with boundary. Our discussion generalizes previous work by Ecker, Huisken, Gerhardt and others with respect to a crucial aspects:
we consider any non-compact Cauchy hypersurface under the assumption of bounded geometry.
Moreover, we specialize the aforementioned works by considering globally hyperbolic Lorentzian space-times equipped with a specific class of warped product metrics.
\end{abstract}

\tableofcontents

\section{Introduction and statement of the main result}

We are interested in maximal space-like Cauchy hypersurfaces, where maximality refers to vanishing mean curvature,
and more generally in space-like Cauchy hypersurfaces with prescribed mean curvature in globally hyperbolic Lorentzian space-times.
These play an important role in gravitational physics, such as in the first proof of the positive mass theorem by Schoen and Yau \cite{SY1, SY2} and the
analysis of the Cauchy problem for asymptotically flat space-times by Choquet-Bruhat and York \cite{Choquet} and Lichnerowicz 
\cite{Lic}. We also refer the reader e.g. to Bartnik \cite{bartnik1986maximal} 
and references therein for an overview. \medskip

Construction of such Cauchy hypersurfaces using the prescribed mean curva\-ture flow
has been pioneered by Ecker and Huisken \cite{ecker1991parabolic}. The prescribed 
mean curvature flow in a semi-Riemannian manifold $(N,\gbar)$ 
is a family of embeddings $F(s)\equiv F(\cdot, s):M\rightarrow N$ of a smooth manifold $M$ with
parameter $s \in [0,T]$ in some interval, satisfying an initial value problem 
\begin{equation}\label{PMCFIntro}
\partial_s F(s)=-(\smc-\mathcal{H})\mu \quad F(s=0)=F_0,
\end{equation}
where $\mathcal{H}:M\rightarrow\mathbb{R}$ is the prescribing function, $\smc$ is the mean curvature of $F(s)M \subset N$
and $F_0$ is some initial embedding. Under the mean curvature flow, for every point $p$ in $M$ the normal velocity at which $F(p,s)$ moves is given 
by the mean curvature of $F(s)(M)$ at $F(s,p)$ minus $\mathcal{H}$. If $\mathcal{H} \equiv 0$, the flow is referred to as 
the (usual) mean curvature flow. 
\medskip

Mean curvature flows have been extensively studied in various scenarios. Though we are rather interested in the Lorentzian 
setting, let us mention some results in case of $M$ being a compact hypersurface of a
Riemannian manifold $N$. 
Mean curvature flows in this setting has been studied  e.g. by Huisken 
\cite{Hui84, Hui90}, Ecker \cite{Eck04}, Colding and Minicozzi \cite{CoMi12}, 
White \cite{Whi05}, Mantegazza \cite{mantegazza2011lecture}) and Smoczyk \cite{smoczyk2012mean}, to cite just a few. 
The list is far from complete. \medskip

Prescribed mean curvature flows in a globally hyperbolic Lorentzian space-time $N$
have been studied for compact hypersurfaces $M$ by Ecker and 
Huisken \cite{ecker1991parabolic} and Gerhardt \cite{gerhardt2000hypersurfaces}.
Without spacial compactness, Ecker \cite{ecker1997interior} proved long time existence and convergence for \eqref{PMCFIntro} 
when $N$ is the Minkowski space-time $\R \times \mathbb{R}^m$. Recently in \cite{kroncke2020mean} the authors proved 
convergence of \eqref{PMCFIntro} with $\mathcal{H}=0$ under the assumption 
$N=\mathbb{R}\times M$ with Lorentzian product metric $\gbar=-dt^2+\widetilde{g}$
and $(M,\widetilde{g})$ being asymptotically flat. 

\subsection{Setting and notation}

Up to isometry,
a globally hyperbolic Lorentzian space-time $(N,\gbar)$ is given by
a product $\mathbb{R}\times M$ with Lorentzian metric 
$$
\gbar= e^{\phi} \Bigl( -dt^2+\widetilde{g}_t\Bigr),
$$
where $M$ is a smooth space-like Cauchy hypersurface,
the natural projection $t:\mathbb{R}\times M \to \R$ is a time-function, 
$\phi \in C^\infty(\R \times M)$ is a smooth function and $\widetilde{g}_t$ 
restricts to a a Riemannian metric on $\{t\} \times M$. This statement
is due to Bernal and Sanchez \cite[Theorem 1.1]{BeSa}. \medskip

In the preceeding work \cite{mio}, the first named author has
studied the prescribed mean curvature flow \eqref{PMCFIntro} 
in the special case where the metric $\gbar$ is given by a warped
product 
\begin{equation}\label{GRWSTMetricIntro}
\gbar=-dt^2+f(t)^2\widetilde{g}
\end{equation}
for some positive smooth function $f:\mathbb{R}\rightarrow\mathbb{R}^+$, bounded away from zero.
Assume that for each $s \in [0,T]$ the embedding $F(s)(M) \subset N$ is a space-like Cauchy hypersurface given by the graph of a function 
$u(s):M\to \R$. Then the flow \eqref{PMCFIntro} can be written as an evolution equation for
$u(s)$. As asserted by \cite[Proposition 3.1]{mio}, the evolution is explicitly given by
\begin{equation}\label{PGMCF}
\begin{split}
&\partial_s u + \Delta u= 
\frac{f'(u)}{f(u)}\left(m+\frac{|\widetilde{\nabla}u|_{\widetilde{g}}^2}{f(u)^2-|\widetilde{\nabla} u|_{\widetilde{g}}^2}\right) 
+\mathcal{H} \, \frac{f(u)}{\sqrt{f(u)^2-|\widetilde{\nabla }u|_{\widetilde{g}}^2}}, \\ &u(0, \cdot)=u_0,
\end{split}
\end{equation} 
where $m= \dim M$, $\widetilde{\nabla}$ is the gradient on $M$ defined by $\widetilde{g}$
and $\Delta$ is the (positive) Laplace Beltrami operator induced by the $s$-dependent 
metric $g=F(s)^*\gbar$, which is Riemannian by the space-likeness assumption. \medskip

\noindent In this paper we will always stay in the following setting:

\begin{setting}\label{setting} Consider the following setting. 
\begin{enumerate}
\item Assume $(M,\widetilde{g})$, to be a stochastically complete Riemannian manifold of bounded geometry. 
Assume furthermore that its embedding $F_0(M) \subset N$ is 
a space-like Cauchy hypersurface given by the graph of a function $u_0:M \to \R$.
\medskip

\item Let $f:\mathbb{R}\rightarrow\mathbb{R}^+$ be smooth, uniformly bounded away from zero,
with uniformly bounded first and second derivatives.  We consider a warped product Lorentzian metric 
(usually referred to as generalized Robertson-Walker metric) on $N=\mathbb{R}\times M$
\begin{equation}
\gbar=-dt^2+f(t)^2\widetilde{g}
\end{equation}  

\item The solution $u=u(\cdot, s)$ to \eqref{PGMCF}, if it exists, defines a family of 
embeddings 
$$
F=F(\cdot, s):M\times [0,T]\rightarrow N, \quad F(p,s) := (p, u(p,s)), p \in M.
$$
The induced family of metrics on $M$ is defined by $g = F^*\overline{g}$.
\end{enumerate}
\end{setting}

We want to point out that bounded geometry of $(M,\widetilde{g})$, 
see Definition \ref{bounded-geom-def}, is required in order to apply parabolic Schauder and Krylov-Safonov estimates;
and it is required in \cite{kroncke2020mean} as well. Stochastic completeness, see \S \ref{O-YSec}, allows for 
applications of the Omori-Yau maximum principle.

\begin{rmk}\label{f-uniform-bound}
We will prove long time existence and convergence of \eqref{PGMCF}
under the assumption that $f$ as well as its derivatives are uniformly bounded and
$f\geq \varepsilon >0$ on $\R$. However, once we deduce existence of a uniformly bounded $u$, a posteriori
uniform bounds on $f$ and its derivatives are not necessary, since $f$ appears in \eqref{PGMCF} only as $f(u)$; thus the only relevant values of $f$ are over the bounded range of $u$. 
\end{rmk}

\noindent We will consistently use the following notation and conventions:

\begin{notation}\label{notation} 
\begin{enumerate}
\item Sometimes we drop the 
$s$-dependence notationally. When referring to the evolution in $s$, we will refer to 
the parameter $s$ as \emph{time} as well. \\[-2mm]

\item The upper script $\sim$, as for $\widetilde{\nabla},\widetilde{\Delta}$, stands for differential operators defined in terms of the metric $\widetilde{g}$ on $M$.
We omit any upper script, as for $\nabla$, $\Delta$, to denote the $s$-dependent operators defined with respect to the induced metric $g=g(s)$ on $M$.
The upper script ${}^-$, as for $\overline{\nabla}$ will be used for operators on $(N,\gbar)$.
\medskip

\item We use ummation convention on repeated indices. 
Latin indices will run in $\{1,\dots,m\}$ while the Greek ones are ranging in $\{0,\dots,m\}$.
Finally, we will write $f(u)$ instead of $f\circ u$; we will write $\partial_i$ instead of 
$\partial/\partial {x^i}$ and, as a convention, we will use $\partial_t$ for $\partial_0$ in $N$.
\medskip

\item We will consider $\Delta$ to be the positive Laplace-Beltrami operator, that is 
\begin{equation}
\Delta u=-\divergence(\nabla u).
\end{equation}
\end{enumerate}
\end{notation}

\subsection{Statement of the main result}

\noindent In this paper we present three main results, one on short time existence of the flow, the second on long time existence, and the third one on convergence. These results will require
varying sets of analytic assumptions, which we now list. 

\begin{assump}\label{assumptions} Consider the classical H\"older spaces $C^{k,\alpha}(M)$ with 
integer $k \in \N_0$ and $\alpha \in (0,1)$, defined with respect to the Riemannian metric $\widetilde{g}$. 
We impose

\begin{enumerate}
\item \textbf{\textup{initial regularity:}} $u_0\in C^{3,\alpha}(M)$ and $\mathcal{H} \in C^{\ell,\alpha}(M)$ with $\ell \geq 2$. \\[-4mm]

\item \textbf{\textup{upper barrier:}} $\smc(s=0)-\mathcal{H} \geq \delta > 0$ for some positive $\delta$. \\[-4mm]

\item \textbf{\textup{Time-like convergence:}} $\ric^N(X,X) > 0$ for any time-like $X \in TN$.

\end{enumerate}
\end{assump} \ \\[-7mm]

\noindent While the initial regularity assumption is natural, the other two assumptions are 
rather restrictive. Still, they already appear in \cite{ecker1991parabolic}, cf. also page 606 therein
for the time-like convergence assumption. Gerhardt \cite{gerhardt2000hypersurfaces} studies 
mean curvature flow without time-like convergence assumption, however the authors did 
not succeed in extending his arguments to the non-compact setting.\medskip

\noindent Our first main result is on the short time existence and it is proved in Theorem \ref{short-time-theorem}. 

\begin{thm}\label{theorem-main-short} Impose Assumptions \ref{assumptions} (1).
Then the solution $u$ to the mean curvature flow \eqref{PGMCF} exists 
with $u \in C^{\ell + 2,\alpha}(M \times [0,T])$ for $T>0$ sufficiently small.
The embeddings $F(s)M$ are space-like Cauchy hypersurfaces in $N$.
\end{thm}

\noindent Our next result concerns the long time existence and it is proved in Theorem \ref{long-time-result} (i). 

\begin{thm}\label{theorem-main-long} Impose Assumptions \ref{assumptions} (1) and (2). 
Then the solution $u$ to the mean curvature flow \eqref{PGMCF} exists for all times
with $u \in C^{\ell + 2,\alpha}(M \times I)$ for any compact intervall $I \subset [0,\infty)$.
Moreover, $F(s)M$ are space-like Cauchy hypersurfaces in $N$.
\end{thm}

\noindent Our final main result is a convergence result and it is proved in Theorem \ref{convergence-theorem}.

\begin{thm}\label{theorem-main} Impose Assumptions \ref{assumptions} (1)-(3) and assume that $M$ is 
the open interior of a compact manifold $\overline{M}$ with boundary $\partial M$. 
Then the solution $u$ to the mean curvature flow \eqref{PGMCF} exists for all times
with $u \in C^{\ell + 2,\alpha}(M \times [0,\infty))$ and converges to $u^* \in C^{\ell+2}(M)$.
The induced embedding is a space-like Cauchy hypersurface in $N$ with mean curvature given by $\mathcal{H}$.
\end{thm}

\begin{rmk}
The strict positivity $\ric^N(X,X) > 0$ for any time-like $X \in TN$ in the time-like convergence assumption can be relaxed. Alternatively, one may only assume $\ric^N(X,X) \geq 0$ for any time-like $X \in TN$,
and require additionally $\mathcal{H} \geq \delta > 0$.
Then the results of Theorem \ref{theorem-main} still hold.
\end{rmk}

\subsection{Distinct arguments due to non-compactness}

We should emphasize here that the arguments in our basic references
\cite{ecker1991parabolic} and \cite{gerhardt2000hypersurfaces} in fact do not simply carry over to the non-compact setting. 
Therefore it might be beneficial for the reader to list those points where the arguments had to be adapted
to the non-compact setting. \bigskip
 
\emph{i)} As is usual in the analysis of geometric flows, a priori estimates are a consequence of the maximum principle. 
In the non-compact setting, we apply the Omori-Yau maximum principle on stochastically complete
manifolds. In particular we prove that the (graphical) mean curvature flow stays stochastically complete. \medskip

\emph{ii)} The a priori $C^0$ estimates, as derived e.g. in \cite{Ger96}, use barrier functions. This approach is not easily 
adapted to the non-compact setting, compare the very intricate barrier function construction \cite{kroncke2020mean}. 
That barrier function argument does not carry over to a general bounded geometry setting in any obvious way.\medskip

\emph{iii)} The a priori $C^2$ estimates, as derived e.g. in \cite{gerhardt2000hypersurfaces}, require certain local coordinates 
around some maximum point. In the non-compact setting we cannot expect the maximum to be attained. Instead one works
with the supremum of a solution, which may lie "at infinity" of the manifold. Thus, a different argument, cf. \cite{ecker1991parabolic}, without using special 
coordinates is necessary. \medskip

\emph{iv)} Convergence of the flow in the compact setting is usually a consequence of a compact embedding of 
H\"older spaces. On manifolds with bounded geometry the embedding $C^{k,\alpha}(M) \subset C^{k,\beta}(M)$
with $\beta < \alpha$, is not necessarily compact. We overcome this difficulty by specializing to the case of manifolds
with boundary, where a similar compact embedding holds in the setting of weighted H\"older spaces.

\subsection{Outline of the paper}

We begin in \S \ref{GRWSTGeoSec} with the geometry of generalized Robertson-Walker space-times 
and their space-like hypersurfaces. In \S \ref{section-schauder} we discuss parabolic Schauder and Krylov Safonov estimates on 
manifolds of bounded geometry. These estimates are applied twice. 
First in order to establish short time existence of the flow, and later to turn a priori estimates into H\"older regularity, concluding long time existence.
In \S \ref{O-YSec} we discuss the Omori-Yau maximum principle on 
stochastically complete manifolds. \medskip

In \S \ref{EvEqSec} we derive the evolution equation of the main object of our 
analysis, the gradient function. The proof of the aforementioned evolution equation will be divided in two 
steps in \S \ref{TimeDerSec} and \S \ref{LapSec}. This is due to a lack of literature about (prescribed) mean 
curvature flows in warped product-type Lorentzian manifolds. Thus all the "classical" evolution equations, 
e.g. in \cite{ecker1991parabolic} had to be re-derived. Evolution equations for the mean curvature and for 
the scalar second fundamental form are derived in \S \ref{EvolutionEquationSec} and \S \ref{EvolutionEquationSec2}.\medskip

Uniform a priori bounds are derived in the subsequent three sections, \S \ref{uniform-section}, \S \ref{Final} and \S \ref{C2EstimatesSec}.
The upper bound in the $C^0$-estimates follows a classical argument, while the lower bound uses a trick to overcome
absence of a lower barrier. For the $C^1$-estimates in \S \ref{Final} we follow Gerhardt's argument, cf. \cite{gerhardt2000hypersurfaces}, to conclude 
that a a space-like prescribed graphical mean curvature flow stays uniformly space-like. From here we deduce
long time existence and convergence in \S \ref{conv-section}.

\subsubsection*{Acknowledgements}

The authors wish to thank the University of Oldenburg for the financial support and hospitality.

\section{Geometry of generalized Robertson-Walker space-times}\label{GRWSTGeoSec}

In this work we are interested in generalized Robertson-Walker space-times,
abbreviated as (GRWST), whose definition we now state explicitly once again, 
before continuing in studying its intrinsic geometry. 

\begin{defn}\label{GRWST}
Let $(M,\widetilde{g})$ be an $m$-dimensional Riemannian manifold. 
A generalized Robertson-Walker space-time (GRWST) is an $(m+1)$-dimensional Lorentzian manifold $(N,\overline{g})$ 
satisfying the following conditions:
\begin{enumerate}
\item there exists a diffeomorphism $\Phi: \mathbb{R}\times M \to N$,
\item there exists $f\in C^\infty(\mathbb{R}, \R^+)$ such that $\gbar$ is a warped product, i.e.
\begin{equation}\label{GRWSTMetric}
\Phi^*\gbar=-dt^2+f(t)^2\widetilde{g}.
\end{equation}
\end{enumerate}
Below we will always identify $(N,\overline{g})$ with $(\mathbb{R}\times M, \Phi^*\gbar)$.
GRWST's are automatically time-oriented, 
i.e. admit a nowhere vanishing time-like vector field $T$. Here, we can 
obviously take $T = \partial_t$. 
\end{defn}

\noindent We continue in the setting of Definition \ref{GRWST} and consider a family of embeddings
$F(\cdot, s):M \rightarrow N$ arising as graphs of a family of functions 
$u(\cdot, s) :M\rightarrow\mathbb{R}$ with $s \in [0,T]$ so that 
\begin{equation}
F(p,s)=(u(p,s),p).
\end{equation}
The induced metric on $M$ is given by $g=F^*\gbar$ and is 
explicitly determined in terms of $u$ and $\widetilde{g}$, as
asserted by the next lemma, cf. \cite[Proposition 2.2]{mio} for the proof.

\begin{lem}
The induced metric $g=F^*\gbar$ is given in local coordinates by
\begin{equation}\label{indmetr}
g_{ij}=-u_iu_j+f(u)^2 \widetilde{g}_{ij}.
\end{equation}
The inverse of the metric tensor $g$ can be locally expressed as 
\begin{equation}\label{invindmet}
g^{ij}=\frac{1}{f(u)^2}\widetilde{g}^{ij}+\frac{1}{f(u)^2}\frac{\widetilde{g}^{jl}u_l\widetilde{g}^{im}u_m}{f(u)^2-|\widetilde{\nabla} u|^2}.
\end{equation}
\end{lem}

The prescribed mean curvature flow is a family of metrics on $M$, embedded into $N$
as a space-like graphs, satisfying some mean curvature flow evolution equation. In order
to be precise, we need to gather some geometric quantities and present some useful facts about 
graphical space-like hypersurfaces.

\subsection{Space-like graphs}

A space-like hypersurface of a Lorentzian manifold $(N,\overline{g})$ is a codimension $1$ 
submanifold so that the induced metric is Riemannian. Equivalently, a hypersurface 
is space-like if its unit normal $\mu$ is time-like. We choose, as a convention, that 
the unit normal $\mu$ is future oriented, i.e.
\begin{equation}
- \overline{g} ( T,\mu ) \equiv - \overline{g} ( \partial_t ,\mu ) >0.
\end{equation}
We codify this expression as the \emph{gradient function} in the next definition.
Our analysis will revolve around that gradient function, as e.g. in \cite{ecker1991parabolic}.

\begin{defn}\label{GradFctDef}
Let $A$ be a space-like hypersurface of a time orientable Lorentzian manifold $(N,\overline{g})$
with a nowhere vanishing time-like vector field $T$. The gradient function $v$ is then defined by
\begin{equation}\label{OldGradFct}
v:=-\overline{g} ( T,\mu).
\end{equation}
\end{defn}

We now provide an explicit expression for the gradient function of the hypersurface
$F(M, s) \subset N$ in the GRWST $(N,\overline{g})$. In local coordinates, induced from 
$M$, the (future oriented) unit normal $\mu$ of $F(M, s) \subset N$ is given by (see 
\cite{mio} for more details and computations)
\begin{equation}\label{NOR}
\mu=\frac{f(u)}{\sqrt{f(u)^2-|\widetilde{\nabla}u|^2}} \Bigl( \partial_t+\frac{1}{f(u)^2}\widetilde{g}^{ij}u_j\partial_{i} \Bigr). 
\end{equation}
From there, using $T=\partial_t$, we obtain by Definition \ref{GradFctDef} 
\begin{equation}\label{expressionofv}
v=\frac{f(u)}{\sqrt{f(u)^2-|\widetilde{\nabla}u|^2}}.
\end{equation}

\begin{rmk}
The graph of a function $u:M\rightarrow \mathbb{
R}$ immersed in the GRWST $(N,\gbar)$ is space-like if and only if
\begin{equation}\label{space-like}
|\widetilde{\nabla}u|^2_{\widetilde{g}}<f(u)^2;
\end{equation}
See \S 2 in \cite{mio} for more details.
\end{rmk}

We emphasize that, the geometry induced on $M$ by the embeddings $F\equiv F(\cdot, s)$ 
is different from the geometry arising from the metric $\widetilde{g}$. Therefore we distinguish 
the geometric quantities associated to $g=F^*\gbar$ from those associated to $\widetilde{g}$:
those associated to the latter are indicated by a subscript $\sim$. For instance 
$\nabla$ and $\widetilde{\nabla}$ denote the gradient (or covariant derivative) on $M$ 
with respect to $g$ and $\widetilde{g}$ respectively. We compute
\begin{equation}\label{NormGradMg}
\nabla u=\frac{\widetilde{\nabla}u}{f(u)^2-|\widetilde{\nabla}u|_{\widetilde{g}}^2}, 
\quad
|\nabla u|_g^2=\frac{|\widetilde{\nabla}u|_{\widetilde{g}}^2}{f(u)^2-|\widetilde{\nabla}u|_{\widetilde{g}}^2},
\end{equation} 
where $|\cdot|_g$ denotes the pointwise norm with respect to $g$, while $|\cdot|_{\widetilde{g}}$ refers to the 
pointwise norm with respect to $\widetilde{g}$. From \eqref{expressionofv} and \eqref{NormGradMg} we conclude
the following list of properties for the gradient function. 

\begin{prop}\label{1234} The gradient function $v$ in \eqref{expressionofv}
satisfies the following properties. 
\begin{itemize}
\item[(i)] the gradient function $v$ and the gradient of $u$ are related by 
\begin{equation}\label{v2nablau}
|\nabla u|_g^2=\frac{f(u)^2}{f(u)^2-|\widetilde{\nabla}u|_{\widetilde{g}}^2}-1=v^2-1.
\end{equation}
\item[(ii)] The gradient function $v$ satisfies $v\ge 1$.
\item[(iii)] The pointwise $g$-norm of $\nabla u$ is bounded from above by
\begin{equation}\label{UpperBoundNormGradu}
|\nabla u|^2_g\le v^2.
\end{equation}
\item[(iv)] The following equality holds
\begin{equation}\label{WillBeUsefulForRic}
v^2|\widetilde{\nabla}u|_{\widetilde{g}}^2=f(u)^2|\nabla u|^2_g.
\end{equation}
\end{itemize}
\end{prop}

\subsection{Intrinsic geometry}\label{Pull-BackSubSect}

Consider local coordinates $(x^1,\dots,x^m)$ on $M$, with the corresponding local frame 
$(\partial_1,\dots,\partial_m)$ on $TM$. Identifying $N$ with $\mathbb{R}\times M$, local coordinates on $N$
are given by $(t,x^1,\dots,x^m)$ and a local frame for $TN$ is 
$$
(\partial_t,\partial_1,\dots,\partial_m).
$$
The intrinsic geometry of a GRWST $(N,\overline{g})$ is described completely by the Christoffel symbols,
which are given explicitly by the following formulae.

\begin{lem} 
The Christoffel symbols $\overline{\Gamma}^\alpha_{\beta\eta}$ of the metric tensor $\overline{g}$ over $N$ are given by
\begin{equation}\label{GRWChrisym}
\begin{split}
&\overline{\Gamma}^\alpha_{00}=0, \  \overline{\Gamma}^0_{0i}=0, \ \overline{\Gamma}^k_{ij}=\widetilde{\Gamma}^k_{ij},\\
&\overline{\Gamma}^0_{ij}=f(t)f'(t)\widetilde{g}_{ij}, \ \overline{\Gamma}^k_{0i}=\frac{f(t)f'(t)}{f(t)^2}\delta^k_i.
\end{split}
\end{equation}
\end{lem}

\noindent By making use of the Christoffel symbols listed above, we can compute 
the local expression for the Riemannian curvature tensor on $(N,\overline{g})$. Recall, 
for $\overline{\nabla}$ being the Levi-Civita connection of $TN$, its Riemann curvature
$R^N$ is a $(1,3)$-tensor
$$R^N:\Gamma(TN)\times \Gamma(TN)\times\Gamma(TN)\rightarrow\Gamma(TN).$$
In terms of any coordinate frame $(\partial_\alpha)_\alpha$ for $TN$, the components
\begin{equation}\label{RiemannCurvatureInLocalCoordinates}
{R^N_{\alpha\beta\gamma}}^\delta\partial_\delta=R^N(\partial_\alpha,\partial_\beta)\partial_\gamma.
\end{equation}
can be expressed in terms of Christoffel symbols by
\begin{equation}\label{RTCoord}
R^{N}_{\alpha\beta\gamma}{}^\delta=\overline{\Gamma}_{\beta\gamma,\alpha}^\delta-\overline{\Gamma}_{\alpha\gamma,\beta}^\delta 
+\sum_\eta \overline{\Gamma}_{\beta\gamma}^\eta\overline{\Gamma}_{\alpha\eta}^\delta-
\overline{\Gamma}_{\alpha\gamma}^\eta\overline{\Gamma}_{\beta\eta}^\delta.
\end{equation}
By making use of the metric tensor we can contract indices gaining a $(0,4)$-tensor.
Such a $(0,4)$-tensor will be denoted by $R^N$ as well and is given by
\begin{equation}\label{RiemannCurvatureTensorFormula}
\begin{split}
R^N: & \ \Gamma(TN)\times \Gamma(TN)\times \Gamma(TN)\times\Gamma(TN)\rightarrow C^\infty(N),\\
&R^N(X,Y,Z,W):=\gbar(R^N(X,Y)Z,W), \\
&R^N_{\alpha\beta\gamma\delta}:= R^N(\partial_\alpha,\partial_\beta,\partial_\gamma, \partial_\delta)=
\gbar_{\delta\zeta}{R^N_{\alpha\beta\gamma}}^\zeta.
\end{split}
\end{equation}
Using local coordinates $(t,x^1,\dots,x^m)$, we obtain from \eqref{GRWChrisym}, \eqref{RTCoord}
and \eqref{RiemannCurvatureTensorFormula} the following list of formulas for $R^N$.

\begin{lem}\label{SomeVanishingofRicciLem}
For every $i,j,k\in\{1,\dots,m\}$ we have for the Riemann curvature tensor
\begin{equation}\label{SomeVanishingofRicci}
\begin{split}
&R^N(\partial_t,\partial_i,\partial_j,\partial_k)=  R^N_{0ijk}=0, \\
&R^N(\partial_t,\partial_i,\partial_j,\partial_t)=  R^N_{0ij0}=-f(t)f''(t)\widetilde{g}_{ij}.
\end{split}
\end{equation}
\end{lem}

\noindent We can now compute the values of the Ricci tensor $\ric^N$, defined
in terms of any coordinate frame $(\partial_\alpha)_\alpha$ for $TN$ by
(we sum over $j,k = 1, \dots, m$)
\begin{equation}
\begin{split}
\ric^N(\partial_\alpha, \partial_\beta) 
&= \gbar^{\delta\gamma}R^N(\partial_\delta,\partial_\alpha,\partial_\beta,\partial_\gamma) \\
&= \gbar^{jk}R^N(\partial_{j},\partial_\alpha,\partial_\beta,\partial_{k}) - R^N(\partial_t,\partial_\alpha,\partial_\beta,\partial_t).
\end{split}
\end{equation}

\begin{cor}\label{RicdTdi}
For every $i,j=1,\dots,m$ we have for the Ricci curvature tensor
\begin{equation}\label{Ricciofdt}
\begin{split}
&\ric^N(\partial_t,\partial_t)=-m\frac{f''(t)}{f(t)}, \quad \ric^N(\partial_t,\partial_i)=0, \\
&\ric^N(\partial_i,\partial_j)=\widetilde{\ric}(\partial_i,\partial_j)+f(t)f''(t)\widetilde{g}_{ij}+(m-1)(f'(t))^2\widetilde{g}_{ij}.
\end{split}
\end{equation}
\end{cor}

\subsection{Extrinsic geometry}\label{ExtrinsicGeometrySub}

Consider the graphical embedding $F\equiv F(\cdot, s): M \to N$. 
Its image $F(M)$ is an hypersurface in $N$, i.e. a codimension one submanifold of $N = \R \times M$.
Consider the pullback bundle $F^*TN$ over $M$. The Riemannian metric $\gbar$ induces an inner product 
on $F^*TN$ by setting for any $p\in M$ and any $W_1=(p,w_1),W_2=(p,w_2) \in F_p^*TN$ with $w_1, w_2 \in T_{F(p)}N$
\begin{align}\label{GFTN}
\gbar_p(W_1,W_2) := \gbar_{F(p)}(w_1,w_2).
\end{align}
We will denote the pull-back connection on 
$F^*TN$ (of the Levi-Civita connection on $TN$) by $\PBD$.
It is easy to see, e.g. by computations in local coordinates, that the pull-back connection satisfies the following metric property.
\begin{lem}
For every $X\in \Gamma(TM)$ and for every $Y,Z\in\Gamma(F^*TN)$, one has
\begin{equation}\label{MetricProperty}
X\pr{\overline{g}\pr{Y,Z}}=\overline{g}\pr{\PBD_X Y,Z}+\overline{g}\pr{Y,\PBD_X Z}.
\end{equation}
\end{lem}
\begin{rmk}
Notice that the metric property above holds in general whenever considering the pull-back connection of the Levi-Civita connection; i.e. it does not depend on the topology of $N$ nor in the codimension of $M$.
\end{rmk}

Note that the total differential $DF(p)$ maps $T_pM$ to $F_{p}^*TN$. 
If $\mu$ is the (time-like) unit normal of $F(M)$ and $(\partial_1,\dots,\partial_m)$ denotes the coordinate frame for $TM$ then 
$$
(\mu,DF(\partial_1),\dots,DF(\partial_m)),
$$ 
is a local frame for $F^*TN$. The main object needed for describing the extrinsic geometry of 
a submanifold is the second fundamental form. We recall the usual definitions here briefly and
then specify the results to our GRWST setting.

\begin{defn}
Let $F:M \to (N,\overline{g})$ be an immersion and $g=F^*\gbar$ the metric induced on $M$ 
by pulling-back $\gbar$. Denote by $\nabla$ the Levi-Civita connection on $TM$ associated with the 
induced metric $g$. The second fundamental form is defined for every $X,Y\in\Gamma(TM)$ by
\begin{equation}\label{SFFFormula}
\sff(X,Y):=\PBD_X(DF(Y))-DF(\nabla_X Y).
\end{equation}
\end{defn}

\begin{rmk}
The intuitive definition of the second fundamental form is the comparison 
between the covariant derivatives of the vectors with respect to the two different connections, 
the submanifold one and the ambient space one.
\end{rmk}

The second fundamental form is normal. That is for every $X,Y,Z\in \Gamma(TM)$
\begin{equation}\label{OrthogonalitySFF}
\overline{g}(\sff(X,Y),DF(Z))=0.
\end{equation}
In particular, $\sff(X,Y)$ lies in the $C^\infty(M)-$span of $\mu$. 
We can define the scalar second fundamental form as follows. 

\begin{defn}\label{SSFFDef}
The scalar second fundamental form of $F(M) \subset N$ is a map 
$\ssff:\Gamma(TM)\times \Gamma(TM)\rightarrow\mathbb{R}$ so that for every vector fields $X$ and $Y$ over $M$ 
\begin{equation}\label{ScalarSecondFundamentalForm}
\ssff(X,Y)=\overline{g}(\sff(X,Y),\mu).
\end{equation} 
\end{defn}

\begin{prop}\label{SFF&SSFFNOR}
Let $F(M) \subset N$ be a space-like hypersurface. 
Then, for any $X,Y\in\Gamma(TM)$ we have the 
following relations between $\sff$ and $\ssff$ and between the operator $\PBD \mu:\Gamma(TM)\rightarrow\Gamma(F^*TN)$ and $\ssff$
\begin{equation}\label{SFF=-hnu}
\sff(X,Y)=-\ssff(X,Y)\mu;
\end{equation}
\begin{equation}\label{SOandSFF}
-\gbar\pr{\PBD_X \mu,DF(Y)}=\gbar\pr{\mu,\sff(X,Y)}=\ssff(X,Y).
\end{equation}
\end{prop}
\begin{proof}
Since $(\mu,DF(\partial_1),\dots,DF(\partial_m))$ is a local frame for $F^*TN$, with $\mu$ orthogonal to $DF(\partial_i)$ for every $i$ and  time-like,
$$\sff(X,Y)=-\gbar(\sff(X,Y),\mu)\mu+g^{ij}\gbar(\sff(X,Y),DF(\partial_i))DF(\partial_j)$$
where $g^{ij}$ denotes the inverse of the induced metric on $M$.
Equation \eqref{SFF=-hnu} now follows from \eqref{OrthogonalitySFF}.
\medskip

\noindent
Equation \eqref{SOandSFF} follows from the metric property of $\PBD$, \eqref{SFFFormula},\eqref{SFF=-hnu} and the fact that $\mu$ is time-like.
\end{proof}

\begin{defn}\label{MCVDef}
For an immersion $F:M\rightarrow (N,\overline{g})$ we define 
\begin{enumerate}
\item the mean curvature vector $\vec{H}:=\trace\sff$,
\item the mean curvature $\smc:=\trace \ssff$.
\end{enumerate}
\end{defn}

\noindent From Proposition \ref{SFF&SSFFNOR} we conclude

\begin{cor}\label{TheWantedRelationProvedCOR}
Let $F:M\rightarrow (N,\overline{g})$ be a spacelike hypersurface. Then
$$\vec{H}=-\smc \mu.$$
\end{cor}

Finally, let us state a formula that will be useful later.
From the local expression of the Christoffel symbols $\Gamma_{ij}^k$ 
and of the scalar second fundamental form $\ssff_{ij}$ e.g. equations (2.6) and (2.15) in \cite{mio} we infer
\begin{equation}\label{vhij}
v\ssff_{ij}=-(u_{ij}-\Gamma_{ij}^ku_k)-f(u)f'(u)\widetilde{g}_{ij}.
\end{equation}
\begin{rmk}
We want to point out that equation \eqref{vhij} is exactly the same as equation 
(1.16) in \cite{gerhardt2000hypersurfaces} once substituting the appropriate values 
of the Christoffel symbols of $(N,\gbar)$ expressed in \eqref{GRWChrisym}.
\end{rmk}

\section{Parabolic Schauder and Krylov-Safonov estimates}\label{section-schauder}

In this section we review parabolic Schauder and Krylov-Safonov estimates
on manifolds with bounded geometry. 

\subsection{Classical H\"{o}lder spaces}

Consider a Riemannian manifold $(M,\widetilde{g})$.

\begin{defn} \label{holder}
The H\"{o}lder space $C^{\alpha} \equiv C^{\alpha}(M\times [0,T])$, for $\alpha \in (0,1)$, is defined as the space of continuous functions 
$u \in C^{0}(M\times [0,T])$ which satisfy
\begin{equation}
    [u]_{\alpha} := \sup_{M^2_T} \left\{ \dfrac{|u(p,t)-u(p',t')|}{d(p,p')^{\alpha}+|t-t'|^{\alpha/2}}\right\} < \infty,
\end{equation}
where the supremum is over $M^2_T$ with $M_T:=M \times [0,T]$; the distance $d$ is induced by the metric $\widetilde{g}$.
The H\"{o}lder norm of any $u \in C^{\alpha}(M\times [0,T])$ is defined by
\begin{equation}
    \|u\|_{\alpha} := \|u\|_{\infty} + [u]_{\alpha}.
\end{equation}
\end{defn}

\noindent The resulting normed vector space $C^{\alpha}(M\times [0,T])$ is a Banach space.
As asserted in the next result, cf. \cite[Lemma 2]{bruno} for a similar statement and its proof, an equivalent H\"older norm is 
obtained with spacial and time differences taken only within bounded local regions.

\begin{lem}\label{local-norm}
The following defines an equivalent norm on $C^{\alpha}(M\times [0,T])$
(we will not distinguish equivalent norms notationally)
\begin{equation}\label{local-norm1}
    \|u\|_{\alpha} := \|u\|_{\infty} + [u]'_{\alpha}, \quad 
     [u]'_{\alpha} := \sup_{M^2_{T,\delta}} \left\{ \dfrac{|u(p,t)-u(p',t')|}{d(p,p')^{\alpha}+|t-t'|^{\alpha/2}}\right\},
\end{equation}
where $M^2_{T,\delta} := \{(p,t), (p',t') \in M_T \mid d(p,p')^{\alpha}+|t-t'|^{\alpha/2} \leq \delta\}$.
\end{lem}

\noindent We will only use the H\"older norm $ \|u\|_{\alpha}$ as in \eqref{local-norm1}. 
We also define the higher order H\"{o}lder spaces for any given $k \in \mathbb{N}$ 
in terms of the gradient $\widetilde{\nabla}$ and pointwise norms $| \cdot |_{\widetilde{g}}$ induced by $\widetilde{g}$ by setting
\begin{equation*}
    C^{k,\alpha} \equiv C^{k,\alpha} (M\times [0,T]) := 
    \left\{ u \in C^{k} \left| \
  | \widetilde{\nabla}^{\ell_1} \partial^{\ell_{2}}_{s} u |_{\widetilde{g}} \in C^{\alpha}, \hspace{2mm} \ell_{1} + 2\ell_{2} \leq k \right. \right\}
\end{equation*}
which is a Banach space with the norm
\begin{equation}
\|u\|_{k,\alpha} := \displaystyle\sum_{l_{1} + 2l_{2} \leq k} 
\left\| |\widetilde{\nabla}^{\ell_1} \partial^{\ell_{2}}_{s} u |_{\widetilde{g}} \right\|_{\alpha}.
\end{equation}
We will also use H\"{o}lder spaces for functions depending either only on 
spatial variables or only on the $s$-time variables. We denote the former by $C^{k,\alpha}(M)$ and 
the latter by $C^{k,\alpha}([0,T])$.

\subsection{Manifolds of bounded geometry}

\begin{defn}\label{bounded-geom-def}
We say that a Riemannian manifold $(M,\widetilde{g})$ has bounded geometry if its injectivity radius is bounded away from $0$
and its Ricci curvature is uniformly bounded, i.e. if for any vector field $X$ on $M$ we have $|\widetilde{\ric}(X,X)|\le c \widetilde{g}(X,X)$
for some uniform constant $c>0$.
\end{defn}

The hypothesis of bounded geometry implies in particular that for some $\delta>0$, all (open) balls $B_\delta (x)$
of radius $\delta$, centred at $x \in M$, are uniformly quasi-isometric to the Euclidean ball $B_{\delta}(0) \in \R^m$.
This means that for each $B_\delta (x)$ there exists a diffeomorphism 
$$
\Psi_x: B_\delta (0) \to B_{\delta}(x),
$$
which changes the distances 
at most by a constant factor that can be chosen independently of $x$. 
Using these quasi-isometries $\Psi_x$, we can define in view of Lemma \ref{local-norm} 
an equivalent norm on $C^{k,\alpha} (M\times [0,T])$ as follows. $\Psi_x$
defines a diffeomorphism 
$$
\Psi_x: B_\delta (0)  \times [0,T] \to B_{\delta}(x) \times [0,T].
$$ \ \\[-7mm]

\noindent We denote the H\"older norm on $C^{k,\alpha}(B_\delta (0) \times [0,T])$, defined as in Definition \ref{holder},
by $\| \cdot \|_{k,\alpha,B_\delta(0)\times[0,T]}$ and obtain an equivalent norm on $C^{k,\alpha}(M\times [0,T])$ given by
\begin{equation}\label{local-norm2}
\|u\|_{k,\alpha} = \sup_{x \in M} \left\| \left. \Psi^* u \right|_{B_\delta (x)} \right\|_{k,\alpha,B_\delta(0)\times[0,T]}.
\end{equation}

\subsection{Parabolic Schauder and Krylov-Safonov estimates}

We present the folklore consequence of the classical Krylov-Safonov estimates, see \cite{KS}, 
as well as the classical parabolic Schauder estimates, see \cite{krylov}. We also point out a nice exposition in 
\cite{picard}. We sum up over repeated indices and consider a uniformly elliptic symmetric 
differential operator $L$ acting on $C^\infty_0(M)$.
Here by uniform ellipticity we mean that in local coordinates 
\begin{equation}\label{uniform-elliptic}
\begin{split}
& \Psi^* \circ L \circ (\Psi^*)^{-1}  = - a^{ij} (s,x)\partial_{x_i} \partial_{x_j} + b^{j} (s,x) \partial_{x_j} +c(s,x), \\
&\textup{where} \quad \Lambda^{-1}  \| \xi\|^2 \leq a^{ij}(s,x) \xi_i \xi_j \leq \Lambda \| \xi\|^2, \\
&\textup{and} \qquad \| b(s,x)\| \leq \Lambda^{-1}, 
\ 0 \leq c(s,x) \leq \Lambda^{-1},
\end{split}
\end{equation}
for some uniform $\Lambda > 0$.

\begin{prop}\label{krylov-safonov-lemma}
Consider a uniformly elliptic symmetric differential operator $L$, as in \eqref{uniform-elliptic}, acting on $C^\infty_0(M)$.
Let $\varphi:M\times [0,T] \to \R$ be uniformly bounded and consider a uniformly bounded solution $\omega$ to 
\begin{equation}\label{inhom-eq}
\left( \partial_s + L \right) \omega = \varphi.
\end{equation}
\begin{enumerate}
\item Then there exists a constant $C>0$ depending only on $m$ and $\Lambda$ such that
$$
\| \omega\|_{\alpha} \leq C \Bigl( \|\omega \|_\infty + \| \varphi \|_\infty \Bigr),
$$
\item If additionally, $a^{ij}, b^j, c, \varphi \in C^{k,\alpha}(M\times[0,T])$, then there 
exists a constant $C>0$ depending only on $m, \Lambda$ and the H\"older norms of the local coefficients of $L$, such that
$$
\| \omega\|_{k+2,\alpha} \leq C \Bigl( \|\omega \|_\infty + \| \varphi \|_{k,\alpha} \Bigr).
$$
\end{enumerate}
\end{prop}

\begin{proof} The solution $\omega$ satisfies for each $x\in M$
$$
\left( \partial_s + \Psi^* \circ L \circ (\Psi^*)^{-1} \right) 
\left. \Psi^* \omega \right|_{B_\delta (x)}  = \left. \Psi^* \varphi \right|_{B_\delta (x)}.
$$
Let us set $Q_\delta := B_\delta(0) \times [0,\delta^2]$. By the Krylov-Safonov estimate, 
see \cite[Theorem 4.2]{KS} and cf. \cite[Theorem 12]{picard}, we find for some uniform constant $C>0$, 
depending only on $\delta, m= \dim M$ and the ellipticity constant $\Lambda > 0$ from \eqref{uniform-elliptic}
\begin{align*}
\| \left. \Psi^* \omega \right|_{B_\delta (x)} \|_{\alpha, Q_{\delta/2}} &\leq C \Bigl(\| \left. \Psi^* \omega \right|_{B_\delta (x)}  \|_{\infty,Q_\delta}  
+ \|\left. \Psi^* \varphi \right|_{B_\delta (x)}  \|_{\infty,Q_\delta} \Bigr) \\
&\leq C \Bigl(\| \omega \|_{\infty}  + \| \varphi \|_{\infty} \Bigr).
\end{align*}
Thus, using the H\"older norm in \eqref{local-norm2}, we find
\begin{align*}
\| \omega \|_{\alpha} \leq C \Bigl(\| \omega \|_{\infty}  + \| \varphi \|_{\infty} \Bigr).
\end{align*}
By Lemma \ref{local-norm} we conclude $\omega \in C^{\alpha}(M\times [0,(\delta/2)^2])$.
We extend the regularity statement to the whole time intervall $[0,T]$ (with constants independent of $T$) 
iteratively.
By setting $s= (\delta/2)^2 + s'$, from the argument above we obtain $\omega \in C^{\alpha}(M\times [(\delta/2)^2, 2(\delta/2)^2])$.
The first statement now follows by repeating the iteration, till we reach $T$.
\medskip

For the second statement, standard parabolic Schauder estimates, see \cite[Theorem 8.12.1]{krylov} and cf. 
\cite[Theorem 6]{picard}, assert that  for some uniform constant $C>0$, 
depending only on $\delta, m, \Lambda$ and the H\"older norms of the coefficients
\begin{align*}
\| \left. \Psi^* \omega \right|_{B_\delta (x)} \|_{k+2,\alpha, Q_{\delta/2}} &\leq C \Bigl(\| \left. \Psi^* \omega \right|_{B_\delta (x)}  \|_{\infty,Q_\delta}  
+ \|\left. \Psi^* f \right|_{B_\delta (x)}  \|_{k,\alpha,Q_\delta} \Bigr) \\
&\leq C \Bigl(\| \omega \|_{\infty}  + \| f \|_{k,\alpha} \Bigr).
\end{align*}
By Lemma \ref{local-norm} we conclude $\omega \in C^{k,\alpha}(M\times [0,(\delta/2)^2])$.
Extension to $C^{k,\alpha}(M\times [0,T])$ goes exactly as before.
\end{proof}

We conclude the subsection by presenting some mapping properties for the
parametrix to the inhomogeneous heat equation \eqref{inhom-eq}.
These can be deduced from Proposition \ref{krylov-safonov-lemma} exactly as in \cite[Proposition 10.1]{bruno}, 
cf. \cite[Theorem 8.10.1]{krylov} for the first claim in \eqref{Q-parametrix-mapping} below.

\begin{prop}\label{Q-mapping-properties}
Consider an $s$-independent uniformly elliptic symmetric differential operator $L$ acting on $C^\infty_0(M)$ as above.
The inhomogeneous heat equation $(\partial_s + L) \omega = \varphi$, with $\omega(s=0)=0$ 
and $\varphi \in C^{k,\alpha}(M\times [0,T])$, has a parametrix $Q$
acting as a bounded linear map 
\begin{equation}\label{Q-parametrix-mapping}\begin{split}
&Q:C^{k,\alpha}(M\times [0,T]) \rightarrow C^{k+2,\alpha}(M\times [0,T]), \\
&Q:C^{k+2,\alpha}(M\times [0,T]) \rightarrow s \ C^{k+2,\alpha}(M\times [0,T]).
\end{split}
\end{equation}
\end{prop}

\subsection{Application: short time existence of the flow}

We present a weaker analogue of the main result by
the first named author \cite{mio}, proving short time
existence of the prescribed mean curvature flow with
space-like Cauchy hypersurfaces of bounded geometry.
In contrast, \cite{mio} asserts that starting with a $\Phi$-manifold,
the Cauchy hypersurfaces remain generalized $\Phi$-manifolds
for short time. The following proves Theorem \ref{theorem-main-short}.

\begin{thm}\label{short-time-theorem} Consider Setting \ref{setting} and 
impose Assumptions \ref{assumptions} (1). 
Then the (graphical) mean curvature flow \eqref{PGMCF} exists for $s \in [0,T]$, 
i.e. $u \in C^{\ell +2,\alpha}(M \times [0,T])$ for $T>0$ sufficiently small, and 
the embeddings $F(s)M$ are space-like Cauchy hypersurfaces in $N$.
\end{thm}

\begin{proof}
For convenience of the reader we shall repeat briefly the argument, that 
is worked out in detail in \cite{mio}. We need to linearize the evolution equation 
\eqref{PGMCF}. The most complicated term is the linearization of $\Delta u$. 
Here, $\Delta$ is the Laplace Beltrami operator on $M$ with respect to the
metric $g$ on the graph of $u$, given explicitly by
$$
g_{ij}=-u_iu_j+f(u)^2 \widetilde{g}_{ij}.
$$
From here one computes explicitly 
\begin{equation}\label{laplacian-g}
\begin{split}
\Delta h&=\frac{1}{f(u)^2}\widetilde{\Delta}h +\frac{1}{f(u)^2}\widehat{\Delta}h\\
&+\widetilde{g}(\NT u,\NT h)\frac{1}{f(u)^2(f(u)^2-|\NT u \, |_{\widetilde{g}}^2)}\widetilde{\Delta}u\\
&+\widetilde{g}(\NT u,\NT h)\frac{1}{f(u)^2(f(u)^2-|\NT u\, |_{\widetilde{g}}^2)}\widehat{\Delta}u\\
&-(m-1)\widetilde{g}(\NT u,\NT h)\frac{f'(u)}{f(u)(f(u)^2-|\NT u\, |_{\widetilde{g}}^2)}\\
&+\widetilde{g}(\NT u,\NT h)\frac{f(u)f'(u)}{(f(u)^2-|\NT u\, |_{\widetilde{g}}^2)^2}
\end{split}
\end{equation}
where in the above $\widehat{\Delta}$ is an operator acting on functions over $M$ defined by
\begin{equation}\label{deltahat}
\begin{split}
\widehat{\Delta}h = -\frac{v^2}{f(u)^2}\widetilde{\nabla}^2 h\pr{\widetilde{\nabla}u,\widetilde{\nabla}u}
= -\frac{\widetilde{g}^{jl}u_l\widetilde{g}^{im}u_m}{f(u)^2-|\NT u|_{\widetilde{g}}^2}\left(h_{ij}-\widetilde{\Gamma}_{ij}^kh_k\right),
\end{split}
\end{equation}
where $\widetilde{\nabla}^2$ denotes the Hessian and the second expression is an expression
in local coordinates. Plugging in $u=u_0 + \omega$ we obtain (writing $\widehat{\Delta}_0$ for \eqref{deltahat} with $u_0$ instead of $u$, 
and writing $\Delta_{g(0)}$ for the Laplace Beltrami operator of $g(0)$)
\begin{equation}\label{D1}
\begin{split}
\Delta u&= \Delta_{g(0)} u_0 + \Delta_{g(0)} \omega+ \frac{|\NT u_0|_{\widetilde{g}}^2}{f(u_0)^2(f(u_0)^2-|\NT u_0|_{\widetilde{g}}^2)}
\left( \widetilde{\Delta} \omega + \widehat{\Delta}_0 \omega \right) \\
&+ F'_1(\omega, \NT \omega) + F'_2(\omega, \NT \omega, \NT^2 \omega),
\end{split}
\end{equation}
where $F'_1(\omega, \NT \omega)$ denotes an expression depending at most linearly 
on the entries in brackets, with coefficients given in terms of $u_0, \NT u_0$ and $\NT^2 u_0$.
The summand $F'_2(\omega, \NT \omega, \NT^2 \omega)$ denotes an expression depending at least quadratically
on the entries in brackets, with coefficients given in terms of $u_0, \NT u_0$ and $\NT^2 u_0$.
On the other hand, \eqref{laplacian-g} implies
\begin{equation}\label{D2}
\begin{split}
\Delta_{g(0)}\omega=\frac{1}{f(u_0)^2}\left( \widetilde{\Delta} \omega + \widehat{\Delta}_0 \omega \right)
+F''_1(\omega, \NT \omega),
\end{split}
\end{equation}
where $F''_1(\omega, \NT \omega)$, similar to $F'_1(\omega, \NT \omega)$, denotes an expression depending at most linearly 
on the entries in brackets, with coefficients given in terms of $u_0, \NT u_0$ and $\NT^2 u_0$.
Combining \eqref{D1} and \eqref{D2} we obtain 
\begin{equation}\label{D2.5}
\begin{split}
\Delta u&= \Delta_{g(0)} u_0 + \frac{f(u_0)^2}{(f(u_0)^2-|\NT u_0|_{\widetilde{g}}^2)}
 \Delta_{g(0)} \omega \\
&+ F'_1(\omega, \NT \omega) - f(u_0)^2 F''_1(\omega, \NT \omega)+ F'_2(\omega, \NT \omega, \NT^2 \omega),
\end{split}
\end{equation}
Linearising similarly the remaining terms of \eqref{PGMCF} we obtain
\begin{equation}\label{D3}
\begin{split}
&\Bigl(\partial_s + \frac{f(u_0)^2}{(f(u_0)^2-|\NT u_0|_{\widetilde{g}}^2)} \Delta_{g(0)} \Bigr)\omega \\&
= - \Delta_{g(0)} u_0 +\frac{f'(u_0)}{f(u_0)}\left(m+\frac{|\widetilde{\nabla}u_0|_{\widetilde{g}}^2}{f(u_0)^2-|\widetilde{\nabla} u_0|_{\widetilde{g}}^2}\right) 
+\mathcal{H} \, \frac{f(u_0)}{\sqrt{f(u_0)^2-|\widetilde{\nabla }u_0|_{\widetilde{g}}^2}}
\\ &+ F_1(\omega, \NT \omega) + F_2(\omega, \NT \omega, \NT^2 \omega),
\end{split}
\end{equation}
where $F_1(\omega, \NT \omega)$, similar to $F_1'$ and $F_1''$, denotes an expression depending at most linearly 
on the entries in brackets, with coefficients given in terms of $\mathcal{H}, u_0, \NT u_0$ and $\NT^2 u_0$. Similarly, $F_2(\omega, \NT \omega, \NT^2 \omega)$
denotes an expression depending at least quadratically on the entries in brackets, with coefficients given in terms of $\mathcal{H}, u_0, \NT u_0$ and $\NT^2 u_0$.
\medskip

Assuming $u_0 \in C^{2,\alpha}(M)$, we find in view of \eqref{laplacian-g} and space-likeness
condition \eqref{space-like} that
$$
L :=  \frac{f(u_0)^2}{(f(u_0)^2-|\NT u_0|_{\widetilde{g}}^2)} \Delta_{g(0)},
$$
is uniformly elliptic in the sense of Proposition \ref{Q-mapping-properties}. 
Provided $\mathcal{H} \in C^{\alpha}(M \times [0,T])$, 
the solution $\omega \in C^{2,\alpha}(M \times [0,T])$ to \eqref{D3} is obtained as a fixed point of 
the bounded map
\begin{equation}\label{D3}
\begin{split}
&\Phi: C^{2,\alpha}(M \times [0,T]) \to C^{2,\alpha}(M \times [0,T]), \\
&\omega \mapsto 
Q (F_1(\omega, \NT \omega)+F_2(\omega, \NT \omega, \NT^2 \omega)). 
\end{split}
\end{equation}
The higher regularity assumption $u_0 \in C^{3,\alpha}(M)$ and $\mathcal{H} \in C^{\ell,\alpha}(M)$ with $\ell \geq 1$
implies, exactly as in \cite[Theorem 6.14]{mio}, that the mapping above is a contraction on 
a closed subset of $C^{2,\alpha}(M \times [0,T])$, if $T>0$ is sufficiently small. Thus we
have proved existence of a solution $u \in C^{2,\alpha}(M \times [0,T])$ for $T>0$ sufficiently small.
\medskip

Let us now prove that $u \in C^{3,\alpha}$, where we abbreviate $C^{k,\alpha} \equiv C^{k,\alpha}(M \times [0,T])$. That gain in regularity is not 
a consequence of a fixed point argument, but rather of the Krylov-Safonov estimates in 
Proposition \ref{krylov-safonov-lemma} (ii). More precisely, $u \in C^{2,\alpha}$ implies
in view of \eqref{laplacian-g} that $\Delta u$ is a uniformly elliptic operator with coefficients
being $C^{1,\alpha}$. Moreover, $u \in C^{2,\alpha}$ and $\mathcal{H} \in C^{\ell,\alpha}(M)$ with $\ell \geq 1$
imply that the right hand side of \eqref{PGMCF} is $C^{1,\alpha}$. Thus, applying 
Proposition \ref{krylov-safonov-lemma} (ii) directly to the evolution equation \eqref{PGMCF}
implies that $u \in C^{3,\alpha}$. \medskip

Repeating the argument of the last paragraph allows for bootstrapping:
Even if the initial data is only $u_0 \in C^{3,\alpha}(M)$, we find that, 
provided $\mathcal{H} \in C^{\ell,\alpha}(M)$ for any $\ell \in \N_0$,  $u$ 
admits the following H\"older regularity
$$
u \in C^{\ell+2,\alpha}(M \times [0,T]).
$$ 
\end{proof}

\section{Omori-Yau parabolic maximum principle}\label{O-YSec}

In order to study the behaviour of solutions of parabolic PDEs one usually proceeds by gaining a priori estimates. 
One of the tools employed to obtain such estimates is the parabolic maximum principle.
We will therefore formulate a parabolic maximum principle on stochastically complete 
manifolds. 

\subsection{Stochastically complete manifolds}

A Riemannian manifold $(X,g_X)$ is sto\-chastically complete if the heat kernel $H$ 
of the (positive) Laplace Beltrami operator $\Delta_X$, associated to $g_X$, satisfies
\begin{equation}\label{stoch-compl}
    \int_{X}H(t,p,\widetilde{p})\dvol_{g_X}(\widetilde{p}) = 1.
\end{equation}
Stochastic completeness can be equivalently characterized by a volume growth condition, due to 
Grigor'yan \cite{grigor1986stochastically}, cf. also Theorem 2.11 in \cite{alias2016maximum}.

\begin{thm}\label{Gri}
Let $(X,g_X)$ be a complete Riemannian manifold. Consider for some reference point $p\in X$
the geodesic ball $B_R(p)$ of radius $R$ around $p$. If the function 
\begin{equation}\label{Grig'sFormula}
\frac{R}{\log\left(\vol\left(B_R(p)\right)\right)}\not\in L^1(1,\infty)
\end{equation}
then $(X,g_X)$ is stochastically complete.
\end{thm} 

\subsection{Weak and strong Omori-Yau maximum principles}

The Omori-Yau maximum principle for the Laplacian, defined in e.g. \cite[Definition 2.1]{alias2016maximum} means that for any function $u \in C^{2}(X)$ with 
bounded supremum there is a sequence $\{p_{k}\}_{k} \subset X$ satisfying
\begin{equation}\label{omori-sup-strong}
    u(p_{k}) > \displaystyle\sup_{X}u - \dfrac{1}{k}, \quad |\nabla u(p_k)| \leq \dfrac{1}{k}, \quad - \Delta_{X}u(p_{k}) < \dfrac{1}{k}.
\end{equation}
Similarly, provided $u$ has bounded infimum, there exists a sequence $\{p'_{k}\}_{k} \subset M$ such that 
\begin{equation} \label{omori-inf-strong}
    u(p'_{k}) < \inf_{X} u + \dfrac{1}{k}, \quad |\nabla u(p_k)| \leq \dfrac{1}{k}, \quad- \Delta_{X}u(p'_{k}) > \dfrac{1}{k}.
\end{equation}
By \cite[Theorem 2.3]{alias2016maximum}, the Omori-Yau maximum principle for the Laplacian holds on any 
$(X,g_X)$ with Ricci curvature bounded from below. We shall refer to this principle as the \emph{strong} Omori-Yau maximum
principle in order to distinguish it from another version of the principle on stochastically complete manifolds.

\begin{rmk}
We want to point out a difference with \cite{alias2016maximum} in the 
different sign convention for the Laplace-Beltrami operator.
\end{rmk}

According to \cite[Theorem 2.8 (i) and (iii)]{alias2016maximum}, a similar version of the Omori-Yau maximum principle holds for
stochastically complete manifolds. More precisely, for any $(X,g_X)$ satisfying e.g. the
volume growth condition in Theorem \ref{Gri}, and any function $u \in C^{2}(X)$, there is a sequence $\{p_{k}\}_{k} \subset X$ such that
\begin{equation}\label{omori-sup}
    u(p_{k}) > \displaystyle\sup_{X}u - \dfrac{1}{k} \;\; \mbox{and} \;\; - \Delta_{X}u(p_{k}) < \dfrac{1}{k}.
\end{equation}
Similarly, there exists a sequence $\{p'_{k}\}_{k} \subset M$ such that 
\begin{equation} \label{omori-inf}
    u(p'_{k}) < \inf_{X} u + \dfrac{1}{k} \;\; \mbox{and} \;\; - \Delta_{X}u(p'_{k}) > \dfrac{1}{k}.
\end{equation}

\subsection{An enveloping theorem and applications}

Based on the Omori-Yau maximum principle above, the second named author proved in 
\cite[Proposition 3.1]{bruno} jointly with Caldeira and Hartmann the following enveloping theorem, 
that is formulated for $\Phi$-manifolds but holds on all stochastically complete spaces with
exactly the same proof.

\begin{prop}\label{envelope}
Let $(X,g_X)$ be stochastically complete. Consider any $u \in C^{2,\alpha}(X\times [0,T])$. Then 
$$u_{\sup} (s) := \sup_{X}u(\cdot, s), \quad u_{\inf}(s):= \inf_{X}u(\cdot, s)$$ are locally Lipschitz and differentiable
almost everywhere in $(0,T)$.
Moreover, at those differentiable times $s \in (0,T)$ we find, in the notation of 
\eqref{omori-sup} and \eqref{omori-inf},
\begin{equation}\label{dini}
\begin{split}
&\frac{\partial}{\partial s}u_{\sup} (s) \leq \lim_{\ \epsilon \to \, 0^+} 
\left( \limsup_{k\to \infty} \frac{\partial u}{\partial s} \left(p_k(s+\epsilon), s + \epsilon\right)\right), \\
&\frac{\partial}{\partial s} u_{\inf} (s) \geq \lim_{\ \epsilon \to \, 0^+} 
\left( \liminf_{k\to \infty} \frac{\partial u}{\partial s} \left(p'_k(s+\epsilon), s + \epsilon\right)\right).  
\end{split}
\end{equation}
\end{prop}

Finally we are in the position to prove the parabolic maximum principle for 
stochastically complete manifolds.

\begin{thm}\label{POYMP}
Let $(X,g_X(s))$ be a family of stochastically complete manifolds with $s\in [0,T]$.
Denote the corresponding family of Laplace-Beltrami operators by $\Delta_s$.
Consider solutions $u^\pm \in C^{2,\alpha}(X\times[0,T])$, solving the differential inequalities
\begin{equation}\label{PMPInequality}
\left(\frac{\partial}{\partial s}+\Delta_{s}\right)u^+ \le 0,  \quad
\left(\frac{\partial}{\partial s}+\Delta_{s}\right)u^- \ge 0.
\end{equation}
Then $u^+_{\sup} (s)\le u^+_{\sup}(0)$ and $u^-_{\inf} (s)\ge -_{\inf}(0)$ for every $s\in [0,T]$.
\end{thm}

\begin{proof}
Note first by \eqref{omori-sup} and \eqref{omori-inf}
\begin{align*}
\frac{\partial}{\partial s} u^+\bigl(p_k(s),s\bigr) \leq \frac{1}{k}, \quad
\frac{\partial}{\partial s} u^-\bigl(p'_k(s),s\bigr) \geq -\frac{1}{k}.
\end{align*}
Then in view of Proposition \ref{envelope} we find almost everywhere
$$
\frac{\partial}{\partial t} u^+_{\sup}(t) \leq 0, \quad \frac{\partial}{\partial t} u^-_{\inf}(t) \geq 0.
$$
The claim now follows.
\end{proof}

\subsection{Stochastic completeness along the flow}

In order to apply the Omori-Yau maximum principle along the flow, 
we need to discuss if and under which conditions stochastic completeness holds. 
Let us begin with an easy obervation.

\begin{lem}
Let $(X,g_X)$ be a Riemannian manifold and consider a $(0,2)$-tensor $A$ over $X$.
Its norm $|A|_{g_X}$ with respect to $g_X$ is given in local coordinates by
\begin{equation}\label{TensorNorm}
|A|^2_{g_X}=g_X^{il}g_X^{jq}A_{ij}A_{lq}.
\end{equation} 
Then for any two vector fields $Y$ and $Z$, one has
\begin{equation}\label{TensorNormInequality}
|A(Y,Z)|\le |A|_{g_X}|Y|_{g_X}|Z|_{g_X}.
\end{equation}
\end{lem}

\begin{proof}
Let $(e_i)_i$ be a local orthonormal frame for $g_X$. 
Clearly, equation \eqref{TensorNorm} can be written as 
($m=\dim X$)
$$|A|^2_{g_X}=\sum_{i=1}^m\sum_{j=1}^m A(e_i,e_j)^2.$$
By considering the linear decomposition of $Y$ and $Z$ with respect to the orthonormal frame $(e_i)$, one has 
\begin{align*}
A(Y,Z) =\sum_{i=1}^m g_X(Y,e_i)A(e_i,Z)=\sum_{i=1}^m \sum_{j=1}^m g_X(Y,e_i) g_X(Z,e_j)A(e_i,e_j).
\end{align*} 
A repeated application of the Cauchy-Schwarz inequality leads to the following chain of inequalities
\begin{align*}
|A(Y,Z)|&\le\sum_{i=1}^m \left|g_X(Y,e_i)\right| \sqrt{\sum_{j=1}^m g_X(Z,e_j)^2}\sqrt{\sum_{j=1}^m A(e_i,e_j)^2}\\
&\le \sqrt{\sum_{j=1}^m g_X(Z,e_j)^2} \sqrt{\sum_{i=1}^m g_X (Y,e_i)^2}\sqrt{\sum_{i=1}^m \sum_{j=1}^m A(e_i,e_j)^2}
=  |A|_{g_X}|Z|_{g_X}|Y|_{g_X}.
\end{align*}
\end{proof}

\begin{cor}\label{EquivalentDistancesCor}
Let $(M,g=F^*{\overline{g}})$ be the prescribed graphical mean curvature flow \eqref{PMCFIntro}, 
arising as a family of graphs of functions $u$ over the Riemannian manifold $(M,\widetilde{g})$. 
Let $v$ denote the associated family of gradient functions, as defined in Definition \ref{GradFctDef}.
Then, as long as the flow exists and $v$ is finite, there exist positive $c, C >0$ (depending on $v$) such that
for any $p,q \in M$
\begin{equation}
c \di_{\widetilde{g}}(p,q)\le \di_{g}(p,q)\le C \di_{\widetilde{g}}(p,q).
\end{equation}
In the above $\di_{\widetilde{g}}$ and $\di_g$ denote the distance on $M$ with respect to $\widetilde{g}$ and the 
induced metric $g=F^*\gbar$ respectively.
\end{cor}

\begin{proof}
Let $p$ and $q$ be any two fixed points on $M$ and consider 
a connecting differentiable curve $\gamma:=\gamma(\tau):[0,1]\rightarrow M$ with $\gamma(0)=p$ and $\gamma(1)=q$.
Recall, the curve length with respect to $\widetilde{g}$ is explicitly given by 
$$|\gamma|_{\widetilde{g}}=\int_0^1 \sqrt{\widetilde{g}\pr{\gamma'(\tau),\gamma'(\tau)}}\di \tau.$$
Equation \eqref{TensorNormInequality} and the fact that the distance is by definition the infimum of the lengths of paths joining $p$ and $q$ give
$$\di_{g}(p,q)\le\int_0^1\sqrt{g\pr{\gamma'(\tau),\gamma'(\tau)}}\di \tau \le\sqrt{|g|_{\widetilde{g}}}\int_0^1\sqrt{\widetilde{g}\pr{\gamma'(\tau),\gamma'(\tau)}}\di \tau.$$
Taking infimum over all such paths $\gamma$, we find
$$\di_{g}(p,q) \le\sqrt{|g|_{\widetilde{g}}}\di_{\widetilde{g}}(p,q).$$
The same holds with the roles of $g$ and $\widetilde{g}$ reversed. This leads to 
$$ \frac{1}{\sqrt{|\widetilde{g}|_g}}\di_{\widetilde{g}}(p,q)\le \di_g (p,q) \le\sqrt{|g|_{\widetilde{g}}}\di_{\widetilde{g}}(p,q).$$
The only thing left to do is to estimate $|g|_{\widetilde{g}}$ and $|\widetilde{g}|_g$.
In view of \eqref{TensorNorm}, and keeping in mind that the metric tensor $g$ can be expressed as in \eqref{indmetr}, we compute
in local coordinates
\begin{align*}
|g|^2_{\widetilde{g}}=&\widetilde{g}^{il}\widetilde{g}^{jm}\pr{-u_iu_j+f(u)^2\widetilde{g}_{ij}}\pr{-u_lu_m+f(u)^2\widetilde{g}_{lm}}\\
=&|\NT u|^4_{\widetilde{g}}-2f(u)^2|\NT u|_{\widetilde{g}}^2+f(u)^4m\\
=&(f(u)^2-|\NT u|_{\widetilde{g}})^2+f(u)^4(m-1)=\pr{\frac{f(u)^2}{v^2}}^2+f(u)^4(m-1)\\
\le &f(u)^4 +f(u)^4(m-1)=f(u)^4 m\le c_1.
\end{align*}
In the above the last inequality follows from $v\ge 1$ and assuming that $f$ is uniformly bounded.
With similar arguments we compute, recalling that the inverse of the metric tensor $g$ is expressed as in \eqref{invindmet}
\begin{align*}
|\widetilde{g}|^2_g&=\frac{1}{f(u)^4}\pr{\widetilde{g}^{il}+\frac{\widetilde{g}^{ia}u_a\widetilde{g}^{lb}u_b}{f(u)^2-|\NT u|_{\widetilde{g}}}}\pr{\widetilde{g}^{jm}+\frac{\widetilde{g}^{jc}u_c\widetilde{g}^{mk}u_k}{f(u)^2-|\NT u|_{\widetilde{g}}^2}}\widetilde{g}_{ij}\widetilde{g}_{lm}\\
&=\frac{1}{f(u)^4}\pr{m+2\frac{|\NT u|_{\widetilde{g}}^2}{f(u)^2-|\NT u|^2_{\widetilde{g}}}+\frac{|\NT u|_{\widetilde{g}}^4}{\pr{f(u)^2-|\NT u|_{\widetilde{g}}^2}^2}}\\
&=\frac{1}{f(u)^4}\pr{m-1+\pr{1+\frac{|\NT u|_{\widetilde{g}}^2}{f(u)^2-|\NT u|^2}}^2}\\
&=\frac{1}{f(u)^4}(m-1+v^4)\le c_2 v^4.
\end{align*}
From the above expression we can further notice that $|\widetilde{g}|_g$ is non zero, since $v\ge 1$, thus proving the claim.
\end{proof}

\begin{prop}\label{VolumeFormComparison}
Let $(M,g=g(s))$ be as above in Corollary \ref{EquivalentDistancesCor}.
Then 
\begin{equation}
\dvol_g=\frac{f(u)^m}{v}\dvol_{\widetilde{g}}
\end{equation}
\end{prop}

\begin{proof}
There exists some $\lambda\in C^\infty(M)$ so that $\dvol_g=\lambda\dvol_{\widetilde{g}}$.
By the local expression of the volume form we conclude $\lambda=\sqrt{\det(g\widetilde{g}^{-1})}$.
By expressing the induced metric tensor $g$ in coordinates, cf. \eqref{invindmet}, one has
$$(g\widetilde{g}^{-1})_{ij}=-\widetilde{g}^{jk}u_iu_k+f(u)^2\delta_i^j=f(u)^2\pr{\delta_i^j-\frac{1}{f(u)^2}\widetilde{g}^{kj}u_ku_i}.$$
Let us set $Du^T:=-1/f(u)^2(u_1,\dots,u_m)$, where the lower indices denote partial derivatives with respect to the coordinate frame $(\partial_1,\dots,\partial_m)$.
This implies
\begin{align*}
\det\pr{g\widetilde{g}^{-1}}&=f(u)^{2m}\det\pr{\Id+\NT u\cdot Du^T}=f(u)^{2m}\pr{1+\NT u^T\cdot Du}\\
&=f(u)^{2m}\pr{1-\frac{|\NT u|_{\widetilde{g}}}{f(u)^2}}=\frac{f(u)^{2m}}{v^2}.
\end{align*}
\end{proof}

\begin{prop}
Let $(M,g=g(s))$ be a prescribed graphical mean curvature flow as above in Corollary \ref{EquivalentDistancesCor}.
Assuming that $(M,\widetilde{g})$ is stochastically complete, the flow $(M,g(s))$ stays stochastically complete 
for each fixed $s$, as long as $(M,g(s))$ are space-like, i.e. as long as the gradient function $v(s)$ is finite.
\end{prop}

\begin{proof}
This is a straightforward consequence of Corollary \ref{EquivalentDistancesCor} and Proposition \ref{VolumeFormComparison}.
Together they imply that volume of $R$-balls with respect to $g=g(s)$ and with respect to $\widetilde{g}$ are comparable up to 
constants depending on $v(s)$. Thus, by Theorem \ref{Gri}, $(M,g=g(s))$ is also stochastically complete as long as 
$v(s)$ is bounded, i.e. as long as $(M,g(s))$ are space-like. This proves the claim
\end{proof}

\section{Evolution equation for the gradient function}\label{EvEqSec}

Our central aim is to prove that a space-like prescribed 
graphical mean curvature flow stays uniformly space-like along the flow.
To this end we will prove that the gradient function $v$, defined in \ref{GradFctDef}, 
satisfies a partial differential inequality of the form \eqref{PMPInequality}. 
Such an inequality will follow from the next theorem. 

\begin{thm}\label{EvolutionOfvTHM}
Let $u(s)$ be a solution to the prescribed graphical mean curvature flow 
\eqref{PGMCF} of an $m$-dimensional space-like Cauchy hypersurface. 
Then the gradient function $v\equiv v(s)$ for the graph of $u(s), s\in [0,T]$ satisfies  
the following evolution equation
\begin{equation}\label{EvolutionOfv}
\begin{split}
(\partial_s+\Delta)v=&-\|\ssff\|^2v-\ric^N(\mu,\mu)v-2\frac{f'(u)}{f(u)}\smc+\frac{f'(u)}{f(u)}\mathcal{H}-V(\mathcal{H})\\
&-\frac{f'(u)}{f(u)}\mathcal{H}v^2+2\frac{f'(u)}{f(u)}g(\nabla u,\nabla v)+m\frac{f''(u)}{f(u)}v\\
&-\left(\frac{f'(u)}{f(u)}\right)^2\|\nabla u\|^2 v-\frac{f''(u)}{f(u)}\|\nabla u\|^2 v-m\left(\frac{f'(u)}{f(u)}\right)^2v.
\end{split}
\end{equation} 
In the above $V$ is a vector field over $M$ so that\footnote{Recall that $\gbar$ defines an inner product on
$F^*TN$ by \eqref{GFTN}.}
$$DF(V)=\gbar^{ij}\gbar(\partial_t,DF(\partial_i))DF(\partial_j).$$
\end{thm}

\noindent The above theorem is a direct consequence of the following propositions.

\begin{prop}\label{EvolutionHypAngleProp}
The gradient function $v$ evolves as
\begin{equation}\label{EvolutionHypAngle}
\partial_s v=V(\smc-\mathcal{H})-(\smc-\mathcal{H})\frac{f'(u)}{f(u)}+(\smc-\mathcal{H})\frac{f'(u)}{f(u)}v^2.
\end{equation}
\end{prop}

\begin{prop}\label{ExpansionLapvProp}
The Laplacian of the gradient function $v$ can be expressed as 
\begin{equation}\label{ExpansionLapv}
\begin{split}
\Delta v=&-\frac{f'(u)}{f(u)}\smc-\frac{f'(u)}{f(u)}\smc v^2-V(\smc)+2\frac{f'(u)}{f(u)}g(\nabla u,\nabla v)\\
&-\|\ssff\|^2v-\ric^2(\mu,\mu)v+m\frac{f''(u)}{f(u)}v+\left(\frac{f'(u)}{f(u)}\right)\|\nabla u\|^2v\\
&-\frac{f''(u)}{f(u)}\|\nabla u\|^2v-2\left(\frac{f'(u)}{f(u)}\right)\|\nabla u\|^2v-m\left(\frac{f'(u)}{f(u)}\right)^2v.
\end{split}
\end{equation}
\end{prop} 

\noindent We will prove Proposition \ref{EvolutionHypAngleProp} and \ref{ExpansionLapvProp} in 
\S \ref{TimeDerSec} and \S \ref{LapSec}, respectively.

\subsection{Time derivative of the gradient function}\label{TimeDerSec}

It is important to notice that the unit normal $\mu = \mu(s)$ is a section of an $s$-dependent vector bundle over $M$, namely 
$F^*TN\equiv F(s)^*TN$ with $F(s) \equiv F(\cdot, s): M \to N$ being the graphical embedding given by $F(p, s) = (u(p, s), p)$ for any $p \in M$.
The pull-back connection on $F^*TN$ is denoted by $\PBD$, as introduced in \S \ref{ExtrinsicGeometrySub}.
\medskip

In order to treat the partial derivative $\partial_s$ as a vector field, we consider as in \cite[Section 3.2]{smoczyk2012mean}
the pull-back bundle $\mathcal{F}^*TN = F^*TN \times [0,T]$, where $\mathcal{F} = F: M\times [0,T] \rightarrow N$ is the 
graphical embedding as above, with the parameter $s \in [0,T]$ now considered as a coordinate on $M\times [0,T]$.
In such a way an $s$-derivative becomes a covariant derivative in the direction of $\partial_s$ with respect to 
the pull-back covariant derivative $\SPBD$.
In particular, for every $\sigma\in\Gamma(\mathcal{F}^*TN)$, given in 
local coordinates by $\sigma=\sigma^\alpha\partial_\alpha$ (recall, $\overline{\nabla}$ is the covariant derivative 
of $(N,\gbar)$)
\begin{equation}\label{NewPull-Back2}
\begin{split}
\partial_s\sigma:=&\SPBD_{\partial_s}\sigma=\frac{\partial}{\partial s}
\sigma^\alpha \cdot \partial_\alpha+\sigma^\alpha \cdot \overline{\nabla}_{D\mathcal{F}
(\partial_s)}\partial_\alpha, \\
&\SPBD_{\partial_i}\sigma=\PBD_{\partial_i}\sigma.
\end{split}
\end{equation}

\begin{rmk}
For every $i=1,\dots,m = \dim M$ we obtain for the differential of $\mathcal{F}$
\begin{equation}\label{NewPull-Back1}
\begin{split}
&D\mathcal{F}(\partial_s)=-(\smc-\mathcal{H})\mu, \\ 
&D\mathcal{F}(\partial_i)=DF(\partial_i).
\end{split}
\end{equation}
Due to the symmetry of the second fundamental form associated to $\mathcal{F}$ one has 
\begin{equation}\label{CommutativityCovDer}
\SPBD_{\partial_s}D\mathcal{F}(\partial_{i})=
\SPBD_{\partial_{i}}D\mathcal{F}(\partial_s)
=-\partial_i(\smc-\mathcal{H})\cdot \mu-(\smc-\mathcal{H})\cdot \PBD_{\partial_i}\mu.
\end{equation}
\end{rmk}

\begin{prop}\label{EvolutionNORProp}
For $u$ being a solution of \eqref{PGMCF}, the unit normal $\mu$ evolves as
\begin{equation}\label{EvolutionNOR}
\partial_s \mu:=\SPBD_{\partial_s}\mu=-DF(\nabla(\smc-\mathcal{H})).
\end{equation}
\end{prop}
\begin{proof}
Viewing $\mu$ as a section of the pull-back bundle $\mathcal{F}^*TN$, $\partial_s\mu$ lies in $\Gamma(\mathcal{F}^*TN)$ as well. 
Thus, taking $(\partial_1,  \dots,  \partial_m)$ as a local coordinate frame on $TM$, we get a local frame for $\mathcal{F}^*(TN)$ given by
$(D\mathcal{F}(\partial_s),D\mathcal{F}(\partial_1),\dots,D\mathcal{F}(\partial_m))$.
We can therefore express $\partial_s\mu \equiv \SPBD_{\partial_s}\mu$ with respect to that frame (recall that $\gbar$ defines an inner product on
$F^*TN$ by \eqref{GFTN})
\begin{align*}
\partial_s \mu &= |D\mathcal{F}(\partial_s)|^{-1}_{\gbar} \cdot \gbar\pr{\SPBD_{\partial_s}\mu,D\mathcal{F}(\partial_s)}
D\mathcal{F}(\partial_s) \\ &+g^{ij}\gbar\pr{\SPBD_{\partial_s}\mu,D\mathcal{F}(\partial_{i})}D\mathcal{F}(\partial_{j}).
\end{align*}
From \eqref{NewPull-Back1} the first term reads
\begin{equation}\label{auxiliary1}
\begin{split}
\gbar\pr{\SPBD_{\partial_s}\mu,D\mathcal{F}(\partial_s)}D\mathcal{F}(\partial_s)
&=(\smc-\mathcal{H})^2\gbar\pr{\SPBD_{\partial_s}\mu,\mu}\mu \\
&=\frac{1}{2} (\smc-\mathcal{H})^2 \cdot \partial_s \gbar\pr{\mu,\mu} \cdot \mu = 0,
\end{split}
\end{equation}
where the second equality follows by the metric property of the pull-back connection, and the last equality follows by $\mu$ being of unit length.
Note that $\gbar\pr{\mu,DF(\partial_{i})}=0$, since $\mu$ is normal. 
We conclude again by the metric property of the pull-back connection $\PBD$, and using 
\eqref{CommutativityCovDer} in the second equality
\begin{align*}
\partial_s \mu &= - g^{ij}\gbar\pr{\mu,\PBD_{\partial_s}DF(\partial_{i})}DF(\partial_{j}) \\
&= - g^{ij}\bigl( - \partial_i(\smc-\mathcal{H})\bigr) \gbar\pr{\mu,\mu}DF(\partial_{j}) + g^{ij}(\smc-\mathcal{H})\gbar\pr{\mu, \PBD_{\partial_i}\mu}DF(\partial_{j}) \\
&= - g^{ij}\partial_i(\smc-\mathcal{H}) DF(\partial_{j}) + \frac{1}{2}g^{ij}(\smc-\mathcal{H}) \partial_i \gbar\pr{\mu, \mu}DF(\partial_{j}),
\end{align*}
where we used $\gbar(\mu, \mu)=-1$ in the last equation. 
The second summand vanishes by unitarity of $\mu$, which is a similar argument as in \eqref{auxiliary1}, and thus the
statement follows. 
\end{proof}

\noindent We are now in the position to prove Proposition \ref{EvolutionHypAngleProp}.
\begin{proof}[Proof of Proposition \ref{EvolutionHypAngleProp}]\label{ProofOfTimeEvolGradFct}
Recall that, by definition $v=-\gbar(\mu,\partial_t)$ hence 
\begin{equation}\label{eq3}
\partial_s v=-\gbar\pr{\partial_s \mu,\partial_t}-\gbar\pr{\mu,\PBD_{\partial_s}\partial_t}.
\end{equation}
Formula \eqref{EvolutionNOR} implies that $\partial_s\mu$ lies in $\Gamma(F^*TN)$ and is tangential to the graph of $u$; that is $\gbar\pr{\partial_s\mu,\mu}=0$.
Now $(\mu,DF(\partial_1),\dots,DF(\partial_m))$ is a local frame for $F^*TN$, with $\mu$ orthogonal to 
the other frame elements and time-like. Thus we can write 
\begin{align}\label{ONDDt}
\partial_t = -\gbar(\partial_t,\mu)\mu + \partial_t^{\top}=v\mu+\partial_t^{\top}, \quad \partial_t^{\top}:=g^{ij}\gbar\pr{\partial_t,DF(\partial_i)}DF(\partial_j).
\end{align}
Defining a local vector field $V\in\Gamma(TM)$ by $V= g^{ij}\gbar\pr{\partial_t,DF(\partial_i)} \partial_j$, so that 
$DF(V)=\partial_t^{\top}$, we conclude from Proposition \ref{EvolutionNORProp} (recall $g= F^*\gbar$)
\begin{align*}
\gbar\pr{\partial_s\mu,\partial_t} &=-\gbar\pr{DF\pr{\nabla(\smc-\mathcal{H})}, DF(V)}
\\ &=-g\pr{ \nabla(\smc-\mathcal{H}), V} =-V(\smc-\mathcal{H}).
\end{align*}
For the second term in \eqref{eq3} let us express $\mu$ in the local frame $(\partial_t, \partial_1, \dots, \partial_m)$
\begin{equation}\label{NOROrthogonalDecomposition}
\mu=-\gbar(\mu, \partial_t) \partial_t + \gbar^{ij}\gbar(\mu, \partial_i) \partial_j = v\partial_t+vb^j\partial_{j}
\end{equation}
where $b^j:M\rightarrow\mathbb{R}$, $b^j:=\widetilde{g}^{ij}u_i/f(u)^2$, using \eqref{NOR} in the last equality. 
In particular one writes
$$\SPBD_{\partial_s}\partial_t=-(\smc-\mathcal{H})v\overline{\nabla}_{\partial_t}\partial_t-(\smc-\mathcal{H})vb^j\overline{\nabla}_{\partial_j}\partial_t.$$
From equation \eqref{GRWChrisym} and applying \eqref{NOROrthogonalDecomposition} one concludes
$$\SPBD_{\partial_s}\partial_t=-\frac{f'(u)}{f(u)}(\smc-\mathcal{H})\mu+v\frac{f'(u)}{f(u)}(\smc-\mathcal{H})\partial_t.$$
The result now follows by substituting the above in \eqref{eq3}.
\end{proof}

\subsection{Laplacian of the gradient function}\label{LapSec}
In order to prove Theorem \ref{EvolutionOfvTHM} it remains to 
compute $\Delta v$ (recall $\Delta$ is the Laplacian with respect to $g=F^*\gbar$) at a fixed time $s\in [0,T]$.
For simplicity we will suppress the parameter $s$. Let us consider a local
orthonormal (with respect to $g$) frame $(e_i)_{i}$ of $TM$ over an open neighbourhood $U$,
such that $\nabla_{e_i}e_j(p)=0$ for every $i,j$ at some fixed $p \in M$.
Then we can write for $\Delta v$ at $p$
\begin{equation}\label{Lapv}
\Delta v=e_ie_i\overline{g}(\mu,\partial_t)=e_i(\gbar(\PBD_{e_i}\mu,\partial_t))+e_i(\gbar(\mu,\PBD_{e_i}\partial_t)).
\end{equation}
The second summand in \eqref{Lapv} will be computed using the next proposition (in \eqref{InitialSecondSummand}). 
The first summand is computed below in Lemma \ref{FirstSummandLapvRevisedLemma}. 
\medskip

Let $u$ be a solution of \eqref{PGMCF} and $(e_i)_i$ a local parallel orthonormal frame
at $p \in M$ as above. With respect to a local coordinate frame $(\partial_k)_k$ we can write $e_i=e_i^k\partial_k$ 
for some smooth coefficients $e_i^k:U\rightarrow\mathbb{R}$. Note also that $e_i(u)=-\gbar\pr{DF(e_i),\partial_t}$.
We then obtain the following useful formulae at $p \in M$ (recall the definition in \eqref{SFFFormula} and the fact that
we assumed $\nabla_{e_i}e_j(p)=0$)
\begin{align}
\label{PBDandScalarSecondFundamentalFormNormal}
-\ssff(e_i,e_j)\mu &= \sff(e_i,e_j)=\PBD_{e_i}DF(e_j), \\
\label{NiceFormula}
DF(e_i) &=-\gbar\pr{DF(e_i),\partial_t}\partial_t+e_i^k\partial_k=e_i(u)\partial_t+e_i^k\partial_k.
\end{align}

\begin{prop}\label{useful-formulae}
Let $u$ be a solution of \eqref{PGMCF} and $F$ the corresponding family of graphical embeddings. Let $(e_i)_{i}$ be a local 
orthonormal frame such that $\nabla_{e_i}e_j(p)=0$ for every $i,j$ at some fixed $p \in M$. 
Then at $p$ we have
\begin{enumerate}
\item the covariant derivative of $\partial_t$, as a section of $F^*TN$, can be expressed as 
\begin{equation}\label{CovDerdt}
\begin{split}
\PBD_{e_i}\partial_t&=\frac{f'(u)}{f(u)}DF(e_i)+\frac{f'(u)}{f(u)}\gbar(DF(e_i),\partial_t)\partial_t\\
&=\frac{f'(u)}{f(u)}DF(e_i)-\frac{f'(u)}{f(u)}e_i(u)\partial_t.
\end{split}
\end{equation}
\item For $\mu$ being the unit normal to the graph of $u$ one has 
\begin{equation}\label{InitialSecondSummand}
\gbar\pr{\mu,\PBD_{e_i}\partial_t}=v\frac{f'(u)}{f(u)}e_i(u).
\end{equation}
\item For every $i$ and $j$ ranging between $1$ and $m=\dim M$, one finds
\begin{equation}\label{PreFinalFirstSummand}
e_i\pr{\gbar\pr{\partial_t,DF(e_j)}}=\frac{f'(u)}{f(u)}\delta_{ij}+\frac{f'(u)}{f(u)}e_i(u)e_j(u)+v\ssff(e_i,e_j);
\end{equation}
where $\delta_{ij}$ denotes the Kronecker delta.
In particular, for $i=j$, with the obvious summation convention over repeated indices, one concludes
\begin{equation}\label{ThirdSummandExpansion}
e_i\pr{\gbar\pr{\partial_t,DF(e_i)}}=v\smc+\frac{f'(u)}{f(u)}\pr{m+|\nabla u|_g^2}.
\end{equation}
\end{enumerate} 
\end{prop}

\begin{proof}
\begin{enumerate}
\item From equation \eqref{NiceFormula} we see that 
$$\PBD_{e_i}\partial_t=e_i(u)\overline{\nabla}_{\partial_t}\partial_t+e_i^k\overline{\nabla}_{\partial_k}\partial_t.$$
Equation \eqref{CovDerdt} now follows by substituting the appropriate values of the covariant derivatives on 
the right hand side, described by the Christoffel symbols of $(N,\gbar)$ in \eqref{GRWChrisym}, and using \eqref{NiceFormula} once more.
\item Equation \eqref{InitialSecondSummand} is a direct consequence of \eqref{CovDerdt} and the fact that $\mu$ is normal 
to the graph of $u$, that is $\gbar\pr{\mu,DF(e_i)}=0$ for every $i=1,\dots,m$.
\item The metric property of the pull-back connection $\PBD$ gives
$$e_i\pr{\gbar\pr{\partial_t,DF(e_j)}}=\gbar\pr{\PBD_{e_i}\partial_t,DF(e_i)}+\gbar\pr{\partial_t,\PBD_{e_i}DF(e_j)}.$$
From equations \eqref{CovDerdt} and \eqref{PBDandScalarSecondFundamentalFormNormal} we deduce
\begin{align*}
e_i\pr{\gbar\pr{\partial_t,DF(e_j)}}&=\frac{f'(u)}{f(u)}\gbar\pr{DF(e_i),DF(e_j)}\\
&-\frac{f'(u)}{f(u)}e_i(u)\gbar\pr{\partial_t,DF(e_j)}-\ssff(e_i,e_j)\gbar\pr{\partial_t,\mu}.
\end{align*}
By assumption, $(e_i)_i$ is a local orthonormal frame, with respect to the metric $g=F^*\gbar$, thus $\gbar\pr{DF(e_i),DF(e_j)}=\delta_{ij}$.
Moreover, from equation \eqref{NiceFormula} we compute $-\gbar\pr{\partial_t,DF(e_j)}=e_j(u)$. 
Finally, the result follows by recalling the definition of the gradient function (cf. Definition \ref{GradFctDef}).
\end{enumerate}
\end{proof}
\begin{rmk}\label{NiceRMK}
Notice that equation \eqref{ThirdSummandExpansion} is nothing but (at $p\in M$)
$$\Delta u=-e_i\pr{e_i(u)}=v\smc+\frac{f'(u)}{f(u)}\pr{m+|\nabla u|_g^2}=v\smc+\frac{f'(u)}{f(u)}(m+v^2-1)$$
where we have used (i) in Proposition \ref{1234}. 
This is exactly the same result as Proposition 2.12 in \cite{mio}.
\end{rmk}

\begin{lem}\label{FirstSummandLapvRevisedLemma}
In the notation of Proposition \ref{useful-formulae} we have at $p$
\begin{equation}\label{FirstSummandLapvRevised}
\begin{split}
e_i(\gbar(\PBD_{e_i}\mu,\partial_t))=&-e_i(\gbar(\partial_t,DF(e_j))\ssff(e_i,e_j)
\\&-\gbar(\partial_t,DF(e_j))e_i(\ssff(e_i,e_j)).
\end{split}
\end{equation}
\end{lem}

\begin{proof}
The result follows by the Leibniz rule once we prove that 
\begin{equation}\label{LemFirstSummanEq}
\gbar(\PBD_{e_i}\mu,\partial_t)=-\gbar(\partial_t,DF(e_j))\ssff(e_i,e_j).
\end{equation}
To this end, notice that $\partial_t$, as a section of $F^*TN$, decomposes with respect to the orthonormal frame $\pr{\mu,DF(e_1),\dots,DF(e_m)}$ as
$$\partial_t=-\gbar(\partial_t,\mu)\mu+\gbar(\partial_t,DF(e_j))DF(e_j).$$
Substituting this into the left hand side of \eqref{LemFirstSummanEq} one finds
\begin{align*}
\gbar(\PBD_{e_i}\mu,\partial_t) =-\gbar\pr{\partial_t,\mu}\gbar\pr{\PBD_{e_i}\mu,\mu}+\gbar\pr{\partial_t,DF(e_j)}\gbar\pr{\PBD_{e_i}\mu,DF(e_j)}.
\end{align*}
The first summand now vanishes, since $\mu$ is of unit length. Using \eqref{PBDandScalarSecondFundamentalFormNormal} we now conclude 
at $p \in M$
\begin{align*}
\gbar(\PBD_{e_i}\mu,\partial_t) =-\gbar\pr{\partial_t,DF(e_j)}\ssff(e_i,e_j).
\end{align*}
\end{proof}

\begin{lem}\label{LemmaDerScalMeanCurvature}
In the notation of Proposition \ref{useful-formulae} we have at $p$
\footnote{$\ric^N$ applies to $DF(e_j) \in F^*TN$ similar to \eqref{GFTN}.}
\begin{equation}
e_i(\ssff(e_i,e_j))=e_j(\smc)-\ric^N(DF(e_j),\mu);
\end{equation}
with the obvious summation convention on repeated indices.
\end{lem}
\begin{proof}
From equation \eqref{SOandSFF} in Proposition \ref{SFF&SSFFNOR} we compute, by making use of the metric property of the pull-back connection $\PBD$,
\begin{equation*}
\begin{split}
e_i(\ssff(e_i,e_j))=&-\gbar(\PBD_{e_i}DF(e_i),\PBD_{e_j}\mu)-\gbar(DF(e_i),\PBD_{e_i}\PBD_{e_j}\mu)\\
=&-\gbar(DF(e_i),\PBD_{e_i}\PBD_{e_j}\mu),
\end{split}
\end{equation*}
where the second equality follows from 
the fact that $\gbar(\PBD_{e_i}DF(e_i),\PBD_{e_j}\mu) = 0$
due to \eqref{PBDandScalarSecondFundamentalFormNormal} and the fact that $\mu$ is of unit length, i.e. $\gbar\pr{\mu,\mu}=-1$.
\medskip

Recall that, at $p$, the curvature form of the pull-back connection is the pull-back of the curvature of the connection, that is
$$\PBD_{e_i}\PBD_{e_j}\mu-\PBD_{e_j}\PBD_{e_i}\mu-\PBD_{[e_i,e_j]}\mu=R^N\pr{DF(e_i),DF(e_j)}\mu.$$
Thus, using $\PBD_{[e_i,e_j]}\mu=0$ due to naturality of the pull-back and the computations being performed at $p$, we obtain
(summing over double indices $i$)
\begin{align*}
e_i\pr{\ssff(e_i,e_j)}=&-\gbar\pr{\PBD_{e_i}\PBD_{e_j}\mu ,DF(e_i)}\\
=&-R^N\pr{DF(e_i),DF(e_j),\mu,DF(e_i)}-\gbar\pr{\PBD_{e_j}\PBD_{e_i}\mu,DF(e_i)}\\
=&-\ric^N\pr{DF(e_j),\mu}-e_j\pr{\gbar\pr{\PBD_{e_i}\mu,DF(e_i)}}+\gbar\pr{\PBD_{e_i}\mu,\PBD_{e_j}\mu}\\
=&-\ric^N\pr{DF(e_j),\mu}+e_j(\smc).
\end{align*}
In the above, the second equality is obtained by making use of \eqref{RiemannCurvatureTensorFormula}. 
The first term in the third equality is a consequence of $\pr{\mu,DF(e_1),\dots,DF(e_m)}$ being an orthonormal basis of $T_{F(p)}N$ with $\mu$ time-like.
The second term is instead a mere application of the metric property of the connection $\PBD$.
Finally the fourth identity is the result of the formula \eqref{SOandSFF}, Definition \ref{MCVDef} and of $\mu$ being of unit length.
\end{proof}

We now conclude with the following expression for \eqref{Lapv}.

\begin{prop}\label{FirstLapExpansionProp}
Let $u$ be a solution of \eqref{PGMCF}.
Then the Laplacian of the gradient function $v$ can be expressed in terms of a local 
vector field $V= g^{ij}\gbar\pr{\partial_t,DF(\partial_i)} \partial_j \in\Gamma(TM)$, such that 
$DF(V)=\partial_t^{\top}$, as follows:
\begin{equation}\label{FirstLapExpansion}
\begin{split}
\Delta v=&-\frac{f'(u)}{f(u)}\smc-\frac{f'(u)}{f(u)}\ssff(\nabla u,\nabla u)
-v\|\ssff\|^2-V(H)+\ric^N(DF(V),\mu)\\
&+\frac{f'(u)}{f(u)}g(\nabla u,\nabla v)+\frac{f''(u)}{f(u)}\|\nabla u\|^2v-2\left(\frac{f'(u)}{f(u)}\right)^2\|\nabla u\|^2v\\
&-\frac{f'(u)}{f(u)}\smc v^2-m\left(\frac{f'(u)}{f(u)}\right)^2v.
\end{split}
\end{equation}
\end{prop}
\begin{proof}
Plugging \eqref{InitialSecondSummand}, Lemma \ref{FirstSummandLapvRevisedLemma} and \ref{LemmaDerScalMeanCurvature}
into \eqref{Lapv} yields the following intermediate expression that holds at $p\in M$
\begin{equation}
\begin{split}
\Delta v = & -e_i(\gbar(\partial_t,DF(e_j))\ssff(e_i,e_j)
\\&-\gbar(\partial_t,DF(e_j)) \Bigl( e_j(\smc)-\ric^N(DF(e_j),\mu) \Bigr)
+e_i \Bigl(v\frac{f'(u)}{f(u)}e_i(u)\Bigr).
\end{split}
\end{equation}
Noticing that $V=\gbar\pr{\partial_t,DF(e_k)}e_k$ with summation over $k$, we conclude 
from formula \eqref{PreFinalFirstSummand}, Lemma \ref{FirstSummandLapvRevisedLemma} and Lemma \ref{LemmaDerScalMeanCurvature} 
\begin{equation}\label{almost-there}
\begin{split}
\Delta v = \, & - \Bigl( \frac{f'(u)}{f(u)}\delta_{ij}+\frac{f'(u)}{f(u)}e_i(u)e_j(u)+v\ssff(e_i,e_j) \Bigr) \ssff(e_i,e_j)
\\ &-V(\smc)+\ric^N\pr{DF(V),\mu}
+e_i \Bigl(v\frac{f'(u)}{f(u)}e_i(u)\Bigr) \\
=&-\frac{f'(u)}{f(u)}\smc-\frac{f'(u)}{f(u)}\ssff(\nabla u,\nabla u)-v\|\ssff\|^2\\
&-V(\smc)+\ric^N\pr{DF(V),\mu} +e_i \Bigl(v\frac{f'(u)}{f(u)}e_i(u)\Bigr).
\end{split}
\end{equation}
In order to conclude the statement, it remains to study the last term in 
\eqref{almost-there}. We compute using Remark \ref{NiceRMK}, arriving at an expression 
that holds globally
\begin{align*}
e_i \Bigl(v\frac{f'(u)}{f(u)}e_i(u)\Bigr)= \, &e_i(v)\frac{f'(u)}{f(u)}e_i(u)+ve_i\pr{\frac{f'(u)}{f(u)}}e_i(u)-v\frac{f'(u)}{f(u)}\Delta u\\
= \, &e_i(v)e_i(u)\frac{f'(u)}{f(u)}+v\frac{f''(u)}{f(u)}e_i(u)e_i(u)-v\pr{\frac{f'(u)}{f(u)}}^2e_i(u)e_i(u)\\
&-v^2\frac{f'(u)}{f(u)}\smc-v\pr{\frac{f'(u)}{f(u)}}^2m-v\pr{\frac{f'(u)}{f(u)}}^2|\nabla u|_g^2\\
= \, &g\pr{\nabla u,\nabla v}\frac{f'(u)}{f(u)}+v\frac{f''(u)}{f(u)}|\nabla u|^2_g-2v\pr{\frac{f'(u)}{f(u)}}^2|\nabla u|^2_g\\
&-v^2\frac{f'(u)}{f(u)}\smc-mv\pr{\frac{f'(u)}{f(u)}}^2.
\end{align*}
\end{proof}

\noindent Notice that Proposition \ref{FirstLapExpansionProp} is not yet a proof for Proposition \ref{ExpansionLapvProp}. 
In particular the terms $\ssff\pr{\nabla u,\nabla u}$ and $\ric^N\pr{DF(V),\mu}$ appearing in \eqref{FirstLapExpansion} need to be simplified.

\begin{prop}\label{TheVeryUsefulFormulaLemma}
Let $u$ be a solution of \eqref{PGMCF}. Then
\begin{equation}\label{h(grad,grad)}
\ssff(\nabla u,\nabla u)=-g(\nabla u,\nabla v)-\frac{f'(u)}{f(u)}|\nabla u|^2v.
\end{equation}
\end{prop}
\begin{proof}
Notice that the statement is a direct consequence of a local indentity
\begin{equation}\label{TheVeryUsefulFormula}
v_i=-g^{jk}u_jh_{ki}-\frac{f'(u)}{f(u)}v u_i.
\end{equation}
Indeed, assuming \eqref{TheVeryUsefulFormula} to hold locally, we find
\begin{align*}
g(\nabla u,\nabla v)&=g^{im}u_mv_i
=-g^{im}u_m g^{jk}u_j \ssff_{ik}-g^{im}u_m\frac{f'(u)}{f(u)}v u_i\\
&=-\ssff(\nabla u,\nabla u)-\frac{f'(u)}{f(u)}|\nabla u|_g^2 v.
\end{align*}
This is precisely the statement after rearrangement.
Let us therefore prove \eqref{TheVeryUsefulFormula}.
By making use of \eqref{v2nablau} one has
$$vv_i=\frac{1}{2}\partial_i v^2=\frac{1}{2}\partial_i(1+|\nabla u|^2)=g(\nabla_{\partial_i}\nabla u,\nabla u).$$
Furthermore, $\nabla_{\partial_i}\nabla u$ can be expressed locally as 
$$\nabla_{\partial_i}\nabla u=\partial_i(g^{jk}u_k)\partial_j+g^{jk}u_k\Gamma_{ij}^l\partial_l.$$
By keeping in mind that $\partial_i {g^{jk}} =-g^{jl}\Gamma_{il}^k-\Gamma_{il}^jg^{lk}$, 
one finds 
$$g(\nabla_{\partial_i}\nabla u,\nabla u)=g^{jk}u_j(u_{ki}-\Gamma_{ik}^lu_l).$$
The result now follows by substituting \eqref{vhij} and by noticing that 
$$g^{jk}u_j\widetilde{g}_{ki}=\frac{v^2}{f(u)^2}u_i.$$
\end{proof}

\noindent The only thing left to prove Proposition \ref{ExpansionLapvProp} is a formula for $\ric^N\pr{DF(V),\mu}$. 

\begin{prop}\label{RicTF(V)NORCor}
Let $u$ be a solution for \eqref{PGMCF}
and $F$ the corresponding family of graphical embeddings.
Then we have the following formula for the
vector field $V= g^{ij}\gbar\pr{\partial_t,DF(\partial_i)} \partial_j \in\Gamma(TM)$, with
$DF(V)=\partial_t^{\top}$
\begin{equation}\label{RicTF(V)NOR}
\ric^N(DF(V),\mu)=-\frac{f''(u)}{f(u)}mv-v\ric^N(\mu,\mu).
\end{equation}
\end{prop}
\begin{proof}
By definition of $V$ we have $\partial_t ^\top=DF(V)=\partial_t+\gbar(\partial_t,\mu)\mu$. 
Thus
\begin{align*}
\ric^N(DF(V),\mu)= \, &\ric^N(\partial_t,\mu)-v\ric^N(\mu,\mu)\\
=\, &v\ric^N(\partial_t,\partial_t)+vb^i\ric^N(\partial_t,\partial_i)-v\ric^N(\mu,\mu)\\
=\,  &-vm\frac{f''(u)}{f(u)}+v\ric^N(\mu,\mu).
\end{align*}
The second identity is obtained by considering the orthogonal decomposition of the unit normal 
$\mu$ with respect to the local frame $\pr{\partial_t,\partial_1,\dots,\partial_m}$ of $F^*TN$ 
(cf. formula \eqref{NOROrthogonalDecomposition}); in particular $b^i=\widetilde{g}^{ij}u_j/f(u)^2$. 
The last identity is a consequence of the values for the Ricci tensor described in Corollary \ref{RicdTdi}.
\end{proof}

\section{Evolution equation for the mean curvature}\label{EvolutionEquationSec}

As before, let $u(\cdot, s)$ be a solution to the prescribed graphical mean curvature flow 
\eqref{PGMCF} of an $m$-dimensional space-like Cauchy hypersurface. 
The solution induces a family of embeddings $F(s):M \to N$ with $F(p ,s) = (u(p,s),p)$
for any $p \in M$. The induced metric $g$ is defined by the pullback $g=F(s)^*\gbar$.
We begin with the following basic evolution equations for the metric tensor. 

\begin{prop}\label{EvolutionOfInverse-Prop}
The metric tensor $(g_{ij})$ and its inverse $(g^{ij})$, written in local coordinates, 
satisfy the following evolution equations along \eqref{PGMCF}
\begin{align}
&\partial_s g_{ij}=2(\smc-\mathcal{H})\ssff_{ij}, \label{EvolutioOfMetric} \\
&\partial_s g^{ij}=-2(\smc-\mathcal{H})g^{ik}\ssff_{kl}g^{lj}. \label{EvolutionOfInverse}
\end{align}
\end{prop}

\begin{proof}
Notice that \eqref{EvolutionOfInverse} is a direct consequence of \eqref{EvolutioOfMetric}. 
So we only need to prove \eqref{EvolutioOfMetric}.
Using the same notation as in \S \ref{TimeDerSec} we compute (see right below for the explanation
of the individual steps)
\begin{align*}
\partial_s g_{ij} & = \partial_s\gbar\pr{DF(\partial_i),DF(\partial_j)}\\
&= \gbar\pr{\SPBD_{\partial_s}DF(\partial_i),DF(\partial_j)}+\gbar\pr{DF(\partial_i),\SPBD_{\partial_s}DF(\partial_j)}\\
&= \gbar\pr{\SPBD_{\partial_i}D\mathcal{F}(\partial_s),DF(\partial_j)}+\gbar\pr{DF(\partial_i),\SPBD_{\partial_i}D\mathcal{F}(\partial_s)}\\
&= - \pr{\smc-\mathcal{H}} \Bigl( \gbar (\PBD_{\partial_i}\mu,DF(\partial_j)) + \gbar (DF(\partial_i),\PBD_{\partial_j}\mu)  \Bigr)
= 2\pr{\smc-\mathcal{H}}\ssff_{ij}.
\end{align*}
In the above the first identity follows by definition of the induced metric tensor $g=F^*\gbar$.
The second is just a consequence of the metric property of the pull-back derivative $\SPBD$.
The third line comes from the commutativity $[\partial_s,\partial_i]=0$ of local coordinate fields.
Finally the last equality is a consequence of \eqref{CommutativityCovDer} and the fact that $\mu$ is normal, 
i.e. $\gbar\pr{DF(X),\mu}=0$ for any vector field $X$ over $M$.
\end{proof}

\noindent Next we study the evolution of the scalar second fundamental form.

\begin{prop}\label{EvolutionOfh_ij-Prop}
The tensor $(h_{ij})$ of the scalar second fundamental form 
satisfies the following evolution equation (summing over double indices as usual)
along \eqref{PGMCF}
\begin{equation}\label{EvolutionOfh_ij}
\begin{split}
\partial_s \ssff_{ij} =& \pr{\smc-\mathcal{H}}\pr{g^{kl}\ssff_{ik}\ssff_{jl}-R^N(\mu,DF(\partial_i),DF(\partial_j),\mu)}\\
&+\pr{\smc-\mathcal{H}}_{ij}-\Gamma_{ij}^k\partial_k\pr{\smc-\mathcal{H}}\\
=&\pr{\smc-\mathcal{H}}\pr{g^{kl}\ssff_{ik}\ssff_{jl}-R^N(\mu,DF(\partial_i),DF(\partial_j),\mu)}+\nabla^2_{ij}\pr{\smc-\mathcal{H}}.
\end{split}
\end{equation}
\end{prop}

\begin{proof}
We compute with respect to a local coordinate frame $(\partial_i)$
\begin{align*}
\partial_s h_{ij}&=\partial_s\ssff(\partial_i,\partial_j)=\partial_s\gbar\pr{\PBD_{\partial_i}DF(\partial_j),\mu}\\
&=\gbar\pr{\SPBD_{\partial_s}\PBD_{\partial_i}DF(\partial_j),\mu}+\gbar\pr{\PBD_{\partial_i}DF(\partial_j),\SPBD_{\partial_s}\mu}\\
&=\gbar\pr{\SPBD_{\partial_s}\PBD_{\partial_i}DF(\partial_j),\mu}-h_{ij}\gbar\pr{\mu,\SPBD_{\partial_s}\mu} + \gbar\pr{DF(\nabla_{\partial_i}\partial_j),\SPBD_{\partial_s}\mu}\\
&=\gbar\pr{\SPBD_{\partial_s}\PBD_{\partial_i}DF(\partial_j),\mu}+ \gbar\pr{DF(\nabla_{\partial_i}\partial_j),\SPBD_{\partial_s}\mu}.
\end{align*} 
In the first line we just used \eqref{SOandSFF}.
The second line is obtained by making use of the metric property of the pull-back connection $\SPBD$.
The third is a consequence of \eqref{SFFFormula}.
The last equality follows from the fact that $\mu$ is of unit length, which implies vanishing of the
second term in the third line.
\medskip

\noindent We shall now compute these two terms above. 
By Proposition \ref{EvolutionNORProp} 
\begin{align*}
\gbar\pr{DF(\nabla_{\partial_i}\partial_j),\SPBD_{\partial_s}\mu} &\equiv 
\gbar\pr{DF(\nabla_{\partial_i}\partial_j), \partial_s \mu}  \\ &= 
- g \Bigl( \nabla_{\partial_i}\partial_j , \nabla \pr{\smc-\mathcal{H}} \Bigr)
= -\Gamma_{ij}^k\partial_k\pr{\smc-\mathcal{H}}.
\end{align*} 
This computes the last term. For the first term we proceed as follows. Noting that that $[\partial_i,\partial_s]=0$, we 
obtain from the definition of the Riemann curvature tensor
\begin{equation*}
\SPBD_{\partial_s}\SPBD_{\partial_i}D\mathcal{F}(\partial_j)-\SPBD_{\partial_i}\SPBD_{\partial_s}D\mathcal{F}(\partial_j)
=R^N\pr{D\mathcal{F}(\partial_s),D\mathcal{F}(\partial_i)}D\mathcal{F}(\partial_j).
\end{equation*}  
This implies directly 
\begin{equation}\label{EvOfH-H1}
\begin{split}
&\gbar\pr{\SPBD_{\partial_s}\PBD_{\partial_i}DF(\partial_j),\mu} \\
& = R^N\pr{D\mathcal{F}(\partial_s),D\mathcal{F}(\partial_i),D\mathcal{F}(\partial_j),\mu}+\gbar\pr{\SPBD_{\partial_i}\SPBD_{\partial_s}D\mathcal{F}(\partial_j),\mu}\\
& = -(\smc-\mathcal{H})R^N\pr{\mu,D\mathcal{F}(\partial_i),D\mathcal{F}(\partial_j),\mu}+\partial_i\pr{\gbar\pr{\SPBD_{\partial_j}D\mathcal{F}(\partial_s),\mu}}\\
& \quad -\gbar\pr{\SPBD_{\partial_j}D\mathcal{F}(\partial_s),\SPBD_{\partial_i}\mu},
\end{split}
\end{equation}
where the second equality is a consequence of the \eqref{NewPull-Back1}, \eqref{CommutativityCovDer} and
the metric property of the pull-back connection. Let us now describe the second term on the right hand side of 
the second equation in \eqref{EvOfH-H1}. From \eqref{CommutativityCovDer} we write
\begin{equation}\label{EvOfH-H2}
\begin{split}
&\partial_i\pr{\gbar\pr{\SPBD_{\partial_j}D\mathcal{F}(\partial_s),\mu}} \\ 
&= -\partial_i\Bigl( \partial_j\pr{\smc-\mathcal{H}}\gbar(\mu,\mu) \Bigr)
-\partial_i\pr{\pr{\smc-\mathcal{H}}\gbar\pr{\SPBD_{\partial_j}\mu,\mu}}\\
&=\partial_i \partial_j\pr{\smc-\mathcal{H}},
\end{split}
\end{equation}
where in the second equation we used the fact that $\gbar\pr{\mu,\mu} =-1$.
\medskip

To conclude the computation of \eqref{EvOfH-H1}, we need $\gbar\pr{\SPBD_{\partial_j}D\mathcal{F}(\partial_s),\SPBD_{\partial_i}\mu}$.
To express this we will use equation \eqref{CommutativityCovDer} once more. 
Before presenting the expression let us notice the following. 
In view of \eqref{NewPull-Back1}, $\SPBD_{\partial_i}\mu=\PBD_{\partial_i}\mu$ is a section of the pull-back bundle $F^*TN$. 
Hence it can be linearly decomposed in terms of the local frame $(\mu,DF(\partial_1),\dots,DF(\partial_m))$.
In particular, by keeping in mind that $\mu$ is a unit length time-like vector we conclude 
$$\SPBD_{\partial_i}\mu=g^{jk}\gbar\pr{\PBD_{\partial_i}\mu,DF(\partial_j)}DF(\partial_k)=-g^{jk}\ssff_{ij}DF(\partial_k)$$ 
with the obvious summation over the indices $j$ and $k$. 
Thus we find by \eqref{NewPull-Back1}
\begin{equation}\label{EvOF(H-H)3}
\begin{split}
&\gbar\pr{\SPBD_{\partial_j}D\mathcal{F}(\partial_s),\SPBD_{\partial_i}\mu} \\
& = -\partial_j\pr{\smc-\mathcal{H}} \cdot \gbar\pr{\mu,\PBD_{\partial_i}\mu} -(\smc-\mathcal{H})\gbar\pr{\PBD_{\partial_j}\mu,\PBD_{\partial_i}\mu}\\
& = -(\smc-\mathcal{H}) g^{kl}\ssff_{ik} \ssff_{jl} ,
\end{split}
\end{equation}
where we used $\gbar\pr{\mu,\PBD_{\partial_i}\mu} = 0$ by the metric property of the pull-back connection
and the fact that $\mu$ is of unit length.
Equation \eqref{EvolutionEquationForH-scriptH} now follows by substituting \eqref{EvOF(H-H)3} and \eqref{EvOfH-H2} in \eqref{EvOfH-H1}.
\end{proof}

\begin{cor} The mean curvature evolves along \eqref{PGMCF} by
\begin{equation}\label{EvolutionEquationForH-scriptH}
\begin{split}
(\partial_s + \Delta) (\smc - \mathcal{H}) & = - (\smc - \mathcal{H}) \Bigl( \|\ssff\|^2 + \ric^N(\mu, \mu) \Bigr), \\
(\partial_s + \Delta) (\smc-\mathcal{H})^2 & =-2\pr{\smc-\mathcal{H}}^2\Bigl( \|\ssff\|^2 + \ric^N(\mu, \mu) \Bigr)
\\ & \quad -2|\nabla\pr{\smc-\mathcal{H}}|^2.
\end{split}
\end{equation}
\end{cor}

\begin{proof}
The second evolution equation is a direct consequence of the first one. For the first equation we compute by
Propositions \ref{EvolutionOfInverse-Prop} and \ref{EvolutionOfh_ij-Prop}
\begin{align*}
\partial_s \smc=&\partial_s \pr{g^{ij}\ssff_{ij}}=\partial_s g^{ij}\cdot\ssff_{ij}+g^{ij}\cdot\partial_s\ssff_{ij}\\
=&-2\pr{\smc-\mathcal{H}}\|\ssff\|^2+g^{ij}\partial_s\ssff_{ij}\\
=&-2\pr{\smc-\mathcal{H}}\|\ssff\|^2+\pr{\smc-\mathcal{H}}\pr{\|\ssff\|^2-\ric^N(\mu,\mu)}-\Delta\pr{\smc-\mathcal{H}}.
\end{align*}
\end{proof}

\begin{rmk}
We want to point out a difference between the first equation in \eqref{EvolutionEquationForH-scriptH} 
and the same evolution equation in the proof of \cite[Proposition 4.6]{ecker1991parabolic}.
In the latter one sees an extra term $\gbar\pr{\overline{\nabla}\mathcal{H},\mu}$. Its presence is due to the function $\mathcal{H}$ being defined in 
\cite{ecker1991parabolic} on the ambient Lorentzian manifold $(N,\gbar)$ while in our case $\mathcal{H}$ is defined on $(M,\widetilde{g})$.
In particular, in our case $\partial_s \mathcal{H}$ is just vanishing.
\end{rmk}

\section{Evolution of the scalar second fundamental form}\label{EvolutionEquationSec2}

In this section we derive an evolution equation for the norm (with respect to $g$) of 
the scalar second fundamental form. This will play an essential role for the
uniform $C^0$ and $C^2$-estimates of $u$.  We begin by recalling some useful formulae, 
to be consistent with other references we will also write them in abstract index notation.
\medskip

\noindent First we recall the Codazzi-Mainardi equation, cf. \cite[Theorem 8.9]{lee2018introduction}.

\begin{prop}\label{CodazziEqProp1}
For every $X,Y,Z\in \Gamma(TM)$ one has
\begin{equation}\label{CodazziEq1}
\nabla_X \sff (Y,Z)-\nabla_Y\sff (X,Z)= R^N(DF(X),DF(Y))DF(Z)-DF(R(X,Y)Z).
\end{equation}
\end{prop}

\begin{cor}\label{CodazziEqProp}
For every $X,Y,Z\in \Gamma(TM)$ one has
\begin{equation}\label{CodazziEq}
\nabla \ssff(X,Y,Z)-\nabla\ssff(Y,X,Z)=R^N(DF(X),DF(Y),DF(Z),\mu).
\end{equation}
\end{cor}
\begin{proof}
The result follows by taking the inner product with the unit normal on both sides of \eqref{CodazziEq1} and using the formula for the covariant derivative of tensors.
\end{proof}

\noindent Next we recall Gau\ss' Theorema Egregium.
\begin{thm}
For every $X,Y,Z,W\in \Gamma(TM)$ one has 
\begin{equation}\label{Gauss}
\begin{split}
&R^N\pr{DF(X),DF(Y),DF(Z),DF(W)} \\ &=R^M(X,Y,Z,W)
+\ssff(Y,Z)\ssff(X,W)-\ssff(X,Z)\ssff(Y,W).
\end{split}
\end{equation}
\end{thm}

\begin{proof}
For a proof of this we refer to \cite[Theorem 8.5]{lee2018introduction}, where we used \eqref{SFF=-hnu} and \eqref{ScalarSecondFundamentalForm}.
\end{proof}

\noindent Let $A$ now be a $(0,2)$-tensor over $(M,g)$. Setting for every $X,Y,Z,W\in\Gamma(TM)$
$$\nabla^2 A(X,Y,Z,W)=\nabla (\nabla A)(X,Y,Z,W),$$
one has by direct computation of $\nabla^2 A(X,Y,Z,W)$ and $\nabla^2 A(Y,X,Z,W)$
\begin{equation}
\begin{split}
\nabla^2 A(X,Y,Z,W)-\nabla A(Y,X,Z,W) &=A\pr{R(Y,X)Z,W}+A\pr{Z,R(Y,X)W}\\
&=-A\pr{R(X,Y)Z,W}-A\pr{Z,R(X,Y)W}.
\end{split}
\end{equation}
Following now the same argument as in the proof of \cite[Theorem 10.1]{Zhu02},
we find the following expression for the second derivative of the scalar second fundamental form $\ssff$
\begin{equation}\label{PreSimons}
\begin{split}
\nabla_k\nabla_l \ssff_{ij}=&\nabla_i\nabla_j \ssff_{kl}+g^{pq}R^N_{iklp}\ssff_{qj}+g^{pq}R^N_{ikjp}\ssff_{ql}-\ssff_{kl}R^N_{0ij0}\\
&-\ssff_{kp}g^{pq}R^N_{lijq} -\ssff_{ij}R^N_{k0l0}-\ssff_{ip}g^{pq}R^N_{kjlq}+\nabla_k R^N_{lij0}+\nabla_i R^N_{kjl0}\\
&-g^{pq}\ssff_{kl}\ssff_{ip}\ssff_{qj}+g^{pq}\ssff_{il}\ssff_{kp}\ssff_{qj}-g^{pq}\ssff_{kj}\ssff_{ip}\ssff_{ql}+g^{pq}\ssff_{ij}\ssff_{kp}\ssff_{ql}.
\end{split}
\end{equation}
We want to point out a difference in signs with the result in \cite{Zhu02} due to a different 
sign convention for the scalar second fundamental form $\ssff$ and due to the fact that the ambient space is a Riemannian manifold.
By taking the trace of \eqref{PreSimons} we find the so called Simons identity.

\begin{equation}\label{Simons}
\begin{split}
\Delta \ssff_{ij}=&-\nabla_i\nabla_j\smc-\ssff_{ij}\pr{\|\ssff\|^2+\ric^N(\mu,\mu)}+\smc \ssff_{ik}\ssff_{kj}\\
&+2\ssff_{kl}R^N_{kijl}-\ssff_{pj}R^N_{ikkp}+\ssff_{ip}R^N_{kjkp}\\
&+\smc R^N_{0ij0}-\nabla_k R^N_{kij0}-\nabla_i R^N_{kjk0}.
\end{split}
\end{equation}
By summing the above with \eqref{EvolutionOfh_ij} we find 
\begin{equation}\label{FullEvohij}
\begin{split}
\pr{\partial_s+\Delta}\ssff_{ij}=&-\nabla_i\nabla_j\mathcal{H}-\mathcal{H}\pr{\ssff_{ik}\ssff_{kj}-R^N_{0ij0}}\\
&+2\smc\ssff_{ik}\ssff_{kj}-\ssff_{ij}\pr{\|\ssff\|^2+\ric^N(\mu,\mu)}\\
&+2\ssff_{kl}R^N_{kijl}-\ssff_{jl}R^N_{ikkl}+\ssff_{il}R^N_{kjkl}-\nabla_k R^N_{kij0}-\nabla_i R^N_{kjk0}
\end{split}
\end{equation}

\begin{rmk}
Due to different sign conventions, equation \eqref{FullEvohij} has slight differences in signs with the one in \cite[Proposition 3.2 (i)]{ecker1991parabolic}. 
\end{rmk}

Although the slight change in signs between equation \eqref{FullEvohij} and 
the corresponding one in \cite{ecker1991parabolic} we can conclude by straightforward estimates the same inequality as 
\cite[Proposition 3.2 (iii)]{ecker1991parabolic},  which is the assertion of the final result in this section. 

\begin{prop}
\begin{equation}\label{TheInequality}
\begin{split}
\pr{\partial_s+\Delta}\|\ssff\|^2 &\le -2\|\nabla\ssff\|^2-\|\ssff\|^4+c_0\cdot\pr{1+\|\ssff\|^2+\|\nabla^2\mathcal{H}\|^2} \\
&\le -2\|\nabla\ssff\|^2-\|\ssff\|^4+c_1\cdot\pr{1+\|\ssff\| + \|\ssff\|^2}
\end{split}
\end{equation}
where the constants 
\begin{align*}
&c_0 = c_0\pr{m,v,\|R^N\|,\|\nabla R^N\|,\|\mathcal{H}\|_\infty}, \\
&c_1 = c_1\pr{m,v,\|R^N\|,\|\nabla R^N\|,\|\mathcal{H}\|_\infty, \|\mathcal{H}\|_{C^2}},
\end{align*}
depend on the entries in the brackets.
\end{prop}

\begin{proof} We only indicate the proof idea. We conclude first from \eqref{Simons}
\begin{equation}\label{lapNormhij}
\begin{split}
\Delta \|\ssff\|^2=&-2\|\nabla \ssff\|^2-2\ssff_{ij}\nabla_i\nabla_j\smc-2\|\ssff\|^2\pr{\|\ssff\|^2+\ric^N(\mu,\mu)}\\
&+2\smc\ssff_{ij}\ssff_{jk}\ssff_{ki}+4\ssff_{ij}\ssff_{kl}R^N_{kijl}-2\ssff_{ij}\ssff_{lj}RN_{ikkl}+2\ssff_{ij}\ssff_{il}R^N_{kjkl}\\
&+2\smc\ssff_{ij}R^N_{0ij0}-2\ssff_{ij}\nabla_kR^N_{kij0}-2\ssff_{ij}\nabla_iR^N_{kjk0}.
\end{split}
\end{equation}
With similar arguments by \eqref{EvolutionOfh_ij} we infer
\begin{equation}\label{s-derNormhij}
\begin{split}
\partial_s \|\ssff\|^2=-2\pr{\smc-\mathcal{H}}\pr{\ssff_{ij}\ssff_{jk}\ssff_{ki}+R^N_{0ij0}\ssff_ij}+2\ssff_{ij}\nabla_i\nabla_j\pr{\smc-\mathcal{H}}.
\end{split}
\end{equation}
Summing up \eqref{lapNormhij} and \eqref{s-derNormhij} we find the following evolution equation for the $g$-norm of the scalar second fundamental form.
\begin{equation}\label{EvolutionOfNormh}
\begin{split}
\pr{\partial_s+\Delta}\|\ssff\|^2=&-2\|\nabla\ssff\|^2-2\|\ssff\|^2\pr{\|\ssff\|^2+\ric^N(\mu,\mu)}-2\ssff_{ij}\nabla_i\nabla_j\mathcal{H}\\
&+2\mathcal{H}\pr{\ssff_{ij}\ssff_{jk}\ssff_{ki}+R^N_{0ij0}\ssff_{ij}}+4\ssff_{ij}\ssff_{kl}R^N_{kijl}-2\ssff_{ij}\ssff_{jl}R^N_{ikkl}\\
&+2\ssff_{ij}\ssff_{il}R^N_{kjkl}-2\ssff_{ij}\nabla_k R^N_{kij0}-2\ssff_{ij}\nabla R^N_{kjk0}.
\end{split}
\end{equation}
From here the first inequality follows by bounded geometry. \medskip

For the second inequality, the problem is controlling $\|\nabla^2\mathcal{H}\|^2$. We may want to use the second displayed equation in the proof of 
\cite[Proposition 4.7]{ecker1991parabolic}, however in their setting $\mathcal{H}$ is a function on $N$. Instead, note that
\begin{align*}
\|\nabla^2\mathcal{H}\|^2 & = g^{ik}g^{jl}\nabla^2\mathcal{H}(\partial_i,\partial_j)\nabla^2\mathcal{H}(\partial_k,\partial_l),
\\ \nabla^2\mathcal{H}(\partial_i,\partial_j) & = \partial_i \partial_j \mathcal{H}-\nabla_{\nabla_{\partial_i} \partial_j}\mathcal{H}
=\mathcal{H}_{ij}-\Gamma_{ij}^k\mathcal{H}_k.
\end{align*}
From \cite[(2.6),  (2.15)]{mio} we can conclude
$$\Gamma_{ij}^k\mathcal{H}_k=\widetilde{\Gamma}_{ij}^k\mathcal{H}_k+\frac{f'(u)}{f(u)}\pr{u_i\mathcal{H}_j+u_j\mathcal{H}_i}+\frac{v}{f(u)^2}\widetilde{g}\pr{\widetilde{\nabla}u,\widetilde{\nabla}\mathcal{H}}\ssff_{ij}.$$
Thus we find for some uniform constant $c>0$, using \eqref{indmetr} and uniform bounds on $f$ and its derivatives
$$
\|\nabla^2\mathcal{H}\| \leq c v^2 \Bigl( \|\widetilde{\nabla}^2\mathcal{H}\|_{\widetilde{g}} + 
v^2 \|\ssff\|\|\widetilde{\nabla}\mathcal{H}\|_{\widetilde{g}} \Bigr).
$$
This yields the second inequality and proves the statement.
\end{proof}

\section{$C^0$-estimates: Uniform bounds on the solution $u$}\label{uniform-section}

Consider now a solution $u\in C^{4,\alpha}(M\times[0,T])$ of \eqref{PGMCF}, which
exists for $T>0$ sufficiently small by Theorem \ref{short-time-theorem}, provided
$\mathcal{H} \in C^{2,\alpha}(M)$. We first prove a uniform upper bound.

\begin{prop}\label{upper-bound-u} Consider Setting \ref{setting} and 
impose Assumptions \ref{assumptions} (1), (2). Then $u$ is bounded uniformly from above by $u_{\sup}(0)$.
\end{prop}

\begin{proof}
Notice that the prescribed mean curvature flow \eqref{PGMCF} can be written as 
\begin{equation}\label{uHv}
\partial_s u=-\pr{\smc-\mathcal{H}}v.
\end{equation}
The statement will follow once we prove that $\pr{\smc-\mathcal{H}} \geq 0$. Indeed, due to $v\geq 1$,
$\partial_s u \leq 0$ and thus $u$ is non-increasing with upper bound $u_{\sup}(0)$.\medskip

Since $(\smc - \mathcal{H}) (s=0) > 0$, there exists some maximal interval $[0,\varepsilon) \subseteq [0,T]$ such that $(\smc - \mathcal{H}) (s) > 0$
for $s \in [0,\varepsilon)$. If $\varepsilon = T$, then the right hand side in \eqref{uHv} is negative and the statement follows. 
Let us now assume that $\varepsilon < T$. From \eqref{EvolutionEquationForH-scriptH} we see that
\begin{align*}
(\partial_s+\Delta)(\smc-\mathcal{H})\ge -c(\smc-\mathcal{H}),
\end{align*}
for some positive constant $c$, depending on bounded geometry and $\| h(s)\|$ for $s \in [0,\varepsilon]$.
Since $u \in C^{4,\alpha}(M\times [0,T])$, we note that $\| h(s)\|$ is uniformly bounded for $s \in [0,\varepsilon]$. 
These bounds need not be uniform in $T$ (we have not proved this yet), but this is not necessary for the argument here.
\medskip

Using now the Omori-Yau maximum principle in the form \eqref{omori-inf},
we conclude from the enveloping theorem in Proposition \ref{envelope} that 
$$
\partial_s (\smc-\mathcal{H})_{\inf} \geq -c (\smc-\mathcal{H})_{\inf}.
$$
Integrating this differential inequality yields for $s \in [0,\varepsilon]$
$$
(\smc-\mathcal{H})_{\inf}(s) \geq e^{-cs}  (\smc-\mathcal{H})_{\inf}(0).
$$
Thus, $(\smc-\mathcal{H})(s=\varepsilon) >0$ and hence by maximality of the interval $[0,\varepsilon)$,
we conclude that $\varepsilon = T$ and thus $(\smc - \mathcal{H}) > 0$ on $M\times [0,T]$.
The statement now follows from \eqref{uHv}.
\end{proof}

For the uniform lower bounds the following lemma is useful. 

\begin{lem}\label{UnifromBoundForPowers}
Let $\theta\in C^{2,\alpha}(M\times[0,T])$. 
If $\theta$ satisfies the differential inequality 
\begin{equation}
\pr{\partial_s+\Delta}\theta \le -a^2\theta^2+b,
\end{equation}
with $a>0$ and $b$ constants, then $\theta$ is uniformly bounded from above.
\end{lem}

\begin{proof}
We begin by noticing in the inequality above we can always replace $b$ 
by some positive non-zero $b^2>0$, which we henceforth do. 
Furthermore, if $\sup_M\theta(s) \le b/a$ there is nothing to prove.
Hence let us assume there exists some $s_0\in [0,T]$ so that $\sup_M \theta(s_0) > b/a$.
\medskip

Since $\theta\in C^{2,\alpha}(M\times[0,T])$, from Proposition \ref{envelope}, 
$\theta_{\sup}(s)$ is a locally Lipschitz function and hence positive in a neighbourhood $(s_1,s_2)\subset [0,T]$ containing $s_0$.
Let us then consider the minimal such $s_1\ge 0$. 
Now, by continuity either $s_1=0$ or $\theta_{\sup}(s_1)=b/a$.
\medskip

Let $s\in(s_1,s_2)$ and $\pr{p_k(s)}_k\subset M$ a sequence satisfying the estimates \eqref{omori-inf} for the Omori-Yau maximum principle.
For $k\in\mathbb{N}$ large enough $\theta\pr{p_k(s),s)}>b/a$.
In particular, at these points, $\theta$ satisfies the differential inequality 
$$(\partial_s+\Delta)\theta\pr{p_k(s),s}\le 0.$$
In conclusion, in view of Theorem \ref{POYMP},
$$\theta(\cdot,s)\le\theta_{\sup}(s)\le\theta_{\sup}(s_1)=\frac{b}{a};$$
thus providing the required uniform upper bound.
\end{proof}

Now we establish a lower bound on $u$ for any finite $T$.

\begin{prop}\label{lower-bound-u-1} Consider Setting \ref{setting} and 
impose Assumptions \ref{assumptions} (1), (2). Then $\|\ssff\|$ and $H$ are uniformly bounded. Moreover, $u$ is bounded uniformly for finite times.
\end{prop}

\begin{proof}
In Theorem \ref{MainTheorem} in the next section we will prove that, as a consequence 
of the upper bound in Proposition \ref{upper-bound-u}, $v$ is uniformly bounded.

\noindent
By playing with binomial formulae we find from inequality \eqref{TheInequality}
\begin{equation}
\begin{split}
\pr{\partial_s+\Delta}\|\ssff\|^2 \le -a^2 \|\ssff\|^4+b^2,
\end{split}
\end{equation}
for some uniform $a,b>0$. Note that $a,b$ depend on $v$ and thus uniform bounds on $v$ from 
Theorem \ref{MainTheorem} below are crucial. By Lemma \ref{UnifromBoundForPowers}, 
we conclude that $\|\ssff\|$ is uniformly bounded and hence $H$ is bounded uniformly as well. 
Thus the right hand side of \eqref{uHv} is uniformly bounded and thus $u$ is bounded 
uniformly for finite times. 
\end{proof}

Deriving a uniform time-independent lower bound for $u$ is harder and is usually done by a barrier argument.
In the non-compact setting the barrier argument is somewhat intricate and we present here
a different approach without using barriers. 

\begin{prop} Consider Setting \ref{setting} and 
impose Assumptions \ref{assumptions} (1)-(3). Then $\| \partial_s u \|_\infty$ is exponentially decreasing. In particular
$u$ is bounded uniformly.
\end{prop}

\begin{proof}
As explained in Proposition \ref{lower-bound-u-1}, $\|\ssff\|$ is uniformly bounded. 
Anticipating uniform space-likeness as asserted in Theorem \ref{MainTheorem}, by the time-like convergence assumption
we may take $\delta'>0$ small enough such that $\ric^N(\mu, \mu) \geq \delta'>0$. By 
\eqref{EvolutionEquationForH-scriptH} we conclude 
\begin{equation}
(\partial_s + \Delta) (\smc - \mathcal{H})^2 \leq - \delta' (\smc - \mathcal{H})^2.
\end{equation}
By the Omori-Yau estimates \eqref{omori-sup} and Proposition. \ref{envelope}, we find
\begin{equation}
\partial_s (\smc - \mathcal{H})^2_{\sup} \leq - \delta' (\smc - \mathcal{H})^2_{\sup}.
\end{equation}
This differential inequality can be integrated and yields 
\begin{equation}\label{H-exp}
0 \leq (\smc - \mathcal{H})^2_{\sup}(s) \leq e^{-s\delta'} (\smc - \mathcal{H})^2_{\sup}(0).
\end{equation}
As already noted in the previous proposition, Theorem \ref{MainTheorem} asserts that as a consequence 
of the upper bound in Proposition \ref{upper-bound-u}, $v$ is uniformly bounded. Thus, by \eqref{uHv}
there exists a uniform constant $c>0$ such that
\begin{equation}
\| \partial_s u\|_\infty \leq c e^{-s\delta'}.
\end{equation}
This proves the statement.
\end{proof}

\section{$C^1$-estimates: Preserving the spacelike property}\label{Final}

In this section we will prove the first main result of this paper, namely that a prescribed mean curvature 
flow stays uniformly space-like for as long as the flow exists, if $u$ is uniformly bounded from above. The argument presented here follows 
in spirit the work of Gerhardt in \cite{gerhardt2000hypersurfaces} and is concluded by an 
application of the parabolic maximum principle. We begin by noticing the following.

\begin{prop}
If the gradient function $v$ is uniformly bounded along the flow \eqref{PGMCF}, then the prescribed mean curvature flow 
\eqref{PGMCF} stays space-like.
\end{prop}
\begin{proof}
Assume there exists some $K>1$ so that $v=v(p,s)\le K$ for every $(p,s)\in M\times [0,T]$.
Note that the requirement $K>1$ follows from Proposition \ref{1234} $(ii)$.
Equation \eqref{expressionofv} implies 
$$
f(u)\le K\sqrt{f(u)^2-|\widetilde{\nabla}u|_{\widetilde{g}}^2},
$$
where $\widetilde{\nabla}u$ is as before the gradient of $u$ with respect to $\widetilde{g}$.
We conclude
$$|\widetilde{\nabla}u|_{\widetilde{g}}^2\le \left(1-\frac{1}{K^2}\right)f(u)^2< f(u)^2.$$
Notice that the above is precisely the condition required for a graph to be space-like as pointed out in \cite[Remark 2.6]{mio}.
\end{proof}

In order to prove that the flow stays space-like, it is therefore enough to prove that the gradient function $v$ is uniformly bounded along the flow.
This is precisely the main conclusion of this section, which we now put as a separate theorem 
\begin{thm}\label{MainTheorem}
Consider the flow \eqref{PGMCF} with $\mathcal{H} \in C^{1,\alpha}(M)$ and solution $u \in C^{2,\alpha}(M\times [0,T])$.
Assume that $u$ is uniformly bounded from above.
Then the gradient function $v$ is uniformly bounded along the flow, with the bound depending only on the 
upper bound of $u$. In particular, the prescribed mean curvature flow stays space-like as long as the flow exists.
\end{thm}

The remainder of the section is concerned with proving Theorem \ref{MainTheorem}. 
Following \cite{gerhardt2000hypersurfaces}, the proof idea consists on a clever application of the parabolic maximum principle. 
Before doing that some preparation is needed.

\subsection{Preliminaries}

First we recall the evolution equation \eqref{PGMCF} for $u$.
\begin{prop}\label{EvolutionOfuProp}
Let $\mathcal{H}:M\rightarrow\mathbb{R}$ be a fixed prescribing function.
Under the prescribed mean curvature flow the function $u$ evolves as
\begin{equation}\label{EvolutionOfu}
(\partial_s+\Delta)u=\mathcal{H}v+\frac{f'(u)}{f(u)}(m+v^2-1).
\end{equation}
\end{prop}

\begin{proof}
This is a just \eqref{PGMCF} written in terms of $v$. 
\end{proof}

We will now proceed by presenting some estimates which will be useful for the proof of Theorem \ref{MainTheorem},
and hold for any given graphical embedding (not necessarily along the \eqref{PGMCF} flow).
\begin{prop}\label{g(du,dv)} Consider Setting \ref{setting}. 
Assume the embeddings $F(M) \equiv F(s)(M)$ are space-like for $s\in [0,T]$.
Recall, $g=F^*\gbar$ denotes the induced metric on $M$, $\ssff$ the scalar second fundamental and $v$ the gradient function.
Then there exists a constant $c>0$ independent of $u$, such that
\begin{equation}
|g(\nabla u,\nabla v)|\le \|\ssff\||\nabla u|_g^2+c|\nabla u|_g^2 v.
\end{equation}
\end{prop}
\begin{proof}
In local coordinates one has
$$g(\nabla u,\nabla v)=g^{ik}u_k v_i.$$
From equation \eqref{TheVeryUsefulFormula} we infer 
\begin{align*}
g(\nabla u,\nabla v) &=-g^{ik}u_k g^{jm}u_m\ssff_{ij}-\frac{f'(u)}{f(u)}g^{im}u_mu_i v
\\ &=-g^{ik}u_k g^{jm}u_m\ssff_{ij}-\frac{f'(u)}{f(u)}|\nabla u|_g^2.
\end{align*}
Recall, by Setting \ref{setting} there exists some constant $c>0$ so that $\left\|f'/f\right\|_\infty \le c$.
Thus we conclude
\begin{align*}
|g(\nabla u,\nabla v)| \le |g^{ik}u_k g^{jm}u_m \ssff_{ij}|+\left|\frac{f'(u)}{f(u)}\right||\nabla u|_g^2v\le  \|\ssff\||\nabla u|_g^2+c|\nabla u|_g^2 v,
\end{align*}
where $\|\ssff\|$ denotes the norm of the scalar second fundamental form $\ssff$ with respect to the metric $g$.
\end{proof}

\noindent Next we present an estimate for $\ric^N(\mu,\mu)$.

\begin{prop}\label{RicciEstimateProp}
We continue as in Proposition \ref{g(du,dv)}.
Then there exists a constant $c>0$ independent of $u$, such that
\begin{equation}\label{RicciEstimate}
|\ric^N(\mu,\mu)|\le cv^2.
\end{equation}
\end{prop}

\begin{proof}
From the local expression of the unit normal $\mu$ in \eqref{NOR} we find
$$\ric^N(\mu,\mu)=v^2\ric^N(\partial_t,\partial_t)+\frac{2v^2}{f(u)^2}\ric^N(\partial_t,\widetilde{\nabla}u)+\frac{v^2}{f(u)^4}\ric^N(\widetilde{\nabla}u,\widetilde{\nabla}u).$$
Proposition \ref{RicdTdi} gives
$$\ric^N(\mu,\mu)=-mv^2\frac{f''(u)}{f(u)}+\frac{2v^2}{f(u)^2}\ric^N(\partial_t,\widetilde{\nabla}u)+\frac{v^2}{f(u)^4}\ric^N(\widetilde{\nabla}u,\widetilde{\nabla}u).$$
The second term vanishes due to Proposition \ref{RicdTdi}.
Again, from Proposition \ref{RicdTdi} we infer for the third term
$$\frac{v^2}{f(u)^4}\ric^N(\widetilde{\nabla}u,\widetilde{\nabla}u)=\frac{v^2}{f(u)^4}\widetilde{\ric}(\widetilde{\nabla}u,\widetilde{\nabla}u)+\frac{f''(u)}{f(u)^3}v^2|\widetilde{\nabla}u|_{\widetilde{g}}^2+(m-1)\frac{f'(u)^2}{f(u)^4}v^2|\widetilde{\nabla}u|_{\widetilde{g}}^2.$$
We plug this back into the expression for $\ric^N(\mu,\mu)$ and conclude from $(iv)$ in Proposition \ref{1234}
$$\ric^N(\mu,\mu)=mv^2\frac{f''(u)}{f(u)}+\frac{v^2}{f(u)^4}\widetilde{\ric}(\widetilde{\nabla}u,\widetilde{\nabla}u)+\frac{f''(u)}{f(u)}|{\nabla}u|_g^2+(m-1)\left(\frac{f'(u)}{f(u)}\right)^2|\nabla u|_g^2.$$
By Setting \ref{setting} there exist some constants $c_1,c_2,c_3>0$ so that 
\begin{equation}\label{AssumptionForRicci}
|f(t)|\ge c_1;\;\;\;\;\left|\frac{f'(t)}{f(t)}\right|\le c_2;\;\;\;\;\left|\frac{f''(t)}{f(t)}\right|\le c_3;\;\;\forall\,t\in\mathbb{R},
\end{equation}
By taking the absolute value and keeping in mind that $(M,\widetilde{g})$ is of bounded geometry (i.e. in particular 
$\widetilde{\ric}(X,X) \leq c_4 \widetilde{g}(X,X)$ for any vector field $X$ and some uniform constant $c_4>0$), we obtain the following estimate
$$|\ric^N(\mu,\mu)|\le m c_3 v^2+\frac{c_4}{f(u)^4}v^2+c_3|\nabla u|_g^2+(m-1)c_2^2|\nabla u|_g^2.$$
The statement now follows by noticing that $|\nabla u|_g^2 \leq v^2$ by 
Proposition \ref{1234} and since $|f(t)| \geq c_1 >0$ is bounded uniformly from below away from zero.
\end{proof}

Next we prove that hypersurfaces of $(N,\gbar)$ arising as graphs of some 
H\"older regular functions satisfy the mean curvature structure condition, cf. \cite[chapter 3]{bartnik1984existence}.

\begin{prop}\label{StructureCondition}
We continue as in Proposition \ref{g(du,dv)}.
Recall, $\smc$ denotes the scalar mean curvature and $h$ the scalar second fundamental form. Then for any $\varepsilon>0$ 
and some uniform constant $c>0$ (independent of $u$) we have
\begin{equation}
| \smc + \ssff(\nabla u,\nabla u) | \leq \varepsilon v \| h\| + c \varepsilon^{-1} v^3.
\end{equation}
\end{prop}

\begin{proof}
At any fixed $(p,s) \in M \times [0,T]$ there exists an orthonormal (with respect to $g$) 
basis $\{e_i\}$ of $h$-eigenvectors, i.e. for the Kronecker delta $\delta_{ij}$
$$
h(e_i,e_j) = h_i \delta_{ij}, \quad g(e_i,e_j) = \delta_{ij}.
$$
With respect to that basis we compute at $(p,s)$ (writing $(\nabla u)_i := g(\nabla u, e_i)$)
\begin{align*}
|\smc + \ssff(\nabla u,\nabla u)| &= \left| \sum_{i=1}^m h_i +  \sum_{i=1}^m h_i (\nabla u)_i^2 \right|
\leq \sum_{i=1}^m \Bigl( \frac{1 +  (\nabla u)_i^2}{\sqrt{\varepsilon v}} \Bigr) \sqrt{\varepsilon v} |h_i|
\\ &\leq \sum_{i=1}^m (v \varepsilon)^{-1} \Bigl( 1 +  (\nabla u)_i^2  \Bigr)^2 + \varepsilon v \sum_{i=1}^m h_i^2
\\ &\leq (v \varepsilon)^{-1} \Bigl( m +  2 | \nabla u |_g^2 + | \nabla u |_g^4 \Bigr) + \varepsilon v \| h\|^2.
\end{align*}
By \eqref{UpperBoundNormGradu} we conclude for some $c>0$ (independent of $u$
and $(p,s)$)
\begin{align*}
\smc + \ssff(\nabla u,\nabla u) \leq c \varepsilon^{-1}  v^3 + \varepsilon v \| h\|^2.
\end{align*}\end{proof}

We will need one last estimate.

\begin{prop}\label{EstimateV(H)}
We continue as in Proposition \ref{g(du,dv)}.
Consider as above the (local) vector field $V$ on $M$, so that $DF(V)=\partial_t^{\top}$.
Then for every function $\mathcal{H}\in C^{1,\alpha}(M)$ there exists some uniform
constant $c>0$ (independent of $u$) such that
\begin{equation}
|V(\mathcal{H})|\le c \| \nabla u \|_g \|\mathcal{H}\|_{1,\alpha}.
\end{equation}
\end{prop}

\begin{proof}
It is easy to see that the condition $DF(V)=\partial_t^{\top}$ gives 
$V=-\frac{\widetilde{\nabla}u}{f(u)^{2}}$.

\noindent Therefore in local coordinates, we obtain using \eqref{NormGradMg} in the last estimate
$$
|V(\mathcal{H}) | = \left| \frac{1}{f(u)^2}\widetilde{g}^{ij}u_i\mathcal{H}_j \right| 
\leq c | \widetilde{\nabla} u |_{\widetilde{g}}\|\mathcal{H}\|_{1,\alpha} \leq c | \nabla u|_g \|\mathcal{H}\|_{1,\alpha},
$$
where we used the fact that $f>0$ is uniformly bounded away from zero.
\end{proof}

We are now ready to prove Theorem \ref{MainTheorem}.

\subsection{Proof of Theorem \ref{MainTheorem}}

We will use the ideas of the argument of \cite{gerhardt2000hypersurfaces} with some adaptations due to non-compact geometry.
In the upcoming computations we will systematically suppress the point $(p,s) \in M \times [0,T]$ from notation.  
We consider some constants $\lambda, \rho > 0$, which we will specify later.
\medskip

Let $\varphi=e^{\rho e^{\lambda u}}$. Assume, without loss of generality that $u>1$, 
if it is not the case we can consider $u+C$ for some constant $C>0$ large enough.
An easy computation gives
\begin{equation}\label{Evphi}
(\partial_s+\Delta)\varphi=-\rho\lambda^2 e^{\lambda u}(1+\rho e^{\lambda u})\varphi |\nabla u|_g^2+\rho\lambda e^{\lambda u}\varphi(\partial_s+\Delta)u.
\end{equation}
Let us now set $w=\varphi v$. Therefore we find (recall $\mu$ is defined in \eqref{NOR})
\begin{align*}
(\partial_s+\Delta)w &= v(\partial_s+\Delta)\varphi+\varphi(\partial_s+\Delta)v-2g(\nabla\varphi,\nabla v) \\
&=v(\partial_s+\Delta)\varphi+\varphi(\partial_s+\Delta)v-2\rho\lambda e^{\lambda u}\varphi g(\nabla u,\nabla v).
\end{align*}
Substituting \eqref{Evphi} and \eqref{EvolutionOfv} in the above, we obtain
\begin{align*}
(\partial_s+\Delta)w=I_1 + I_2.
\end{align*}
where $I_1$ and $I_2$ are explicitly given as follows (recall $\mu$ is defined in \eqref{NOR})
\begin{align*}
I_1:= \ & -\rho\lambda^2e^{\lambda u}(1+\rho e^{\lambda u})|\nabla u|_g^2\varphi v-\|\ssff\|^2\varphi v-V(\mathcal{H})\varphi-2\frac{f'(u)}{f(u)}\smc\varphi 
\\ &-2\left(\rho\lambda e^{\lambda u}-\frac{f'(u)}{f(u)}\right)g(\nabla u,\nabla v)\varphi 
-\left(\frac{f'(u)}{f(u)}\right)^2|\nabla u|_g^2\varphi v,\\
I_2:= \ &\rho\lambda e^{\lambda u}\varphi v(\partial_s+\Delta)u -\ric^N(\mu,\mu)\varphi +\frac{f'(u)}{f(u)}\mathcal{H}\varphi-\frac{f'(u)}{f(u)}\mathcal{H}\varphi v^2 \\ 
&+m\frac{f''(u)}{f(u)}\varphi v-\frac{f''(u)}{f(u)}|\nabla u|_g^2\varphi v-m\left(\frac{f'(u)}{f(u)}\right)^2\varphi v
\end{align*}
First, we estimate $I_2$ from above.
By Setting \ref{setting} there exist some constants $c_1, c_2>0$ such that the warping function $f:\mathbb{R}\rightarrow\mathbb{R}^+$ satisfies for any $t \in \R$
$$|f(t)|\ge c_1, \quad \left|\frac{f'(t)}{f(t)}\right|\le c_2, \quad \left|\frac{f''(t)}{f(t)}\right|\le c_2.$$
From equation \eqref{EvolutionOfu} we now deduce for some $c_3>0$ depending on $\|\mathcal{H}\|_\infty$
$$(\partial_s+\Delta)u\le c_3 v^2.$$
Since $|\nabla u|_g \leq v$ by $(iii)$ in Proposition \ref{1234}, we arrive by 
Propositions \ref{RicciEstimateProp} and \ref{EstimateV(H)} at the following estimate of $I_2$ (we write 
$c>0$ for any uniform positive constant)
\begin{align*}
I_2 \leq c \rho\lambda e^{\lambda u}\varphi v^2 + c |\nabla u|_g^2\varphi v + c |\nabla u|_g \varphi  \leq c \rho\lambda e^{\lambda u}\varphi v^3.
\end{align*}
The estimate of $I_1$ is slightly more involved. 
Using the formula from Proposition \eqref{TheVeryUsefulFormulaLemma} 
\begin{equation*}
\ssff(\nabla u,\nabla u)=-g(\nabla u,\nabla v)-\frac{f'(u)}{f(u)}|\nabla u|^2v,
\end{equation*}
we can rewrite $I_1$ as follows
\begin{align*}
I_1=&-\rho\lambda^2e^{\lambda u}(1+\rho e^{\lambda u})|\nabla u|_g^2\varphi v-\|\ssff\|^2\varphi v-2\frac{f'(u)}{f(u)} \Bigl( \smc + \ssff(\nabla u,\nabla u) \Bigr) \varphi\\
&-3\left(\frac{f'(u)}{f(u)}\right)^2|\nabla u|_g^2 \varphi v - m\left(\frac{f'(u)}{f(u)}\right)^2\varphi v
-2 \rho\lambda e^{\lambda u} g(\nabla u,\nabla v)\varphi.
\end{align*}
By Proposition \ref{StructureCondition} we find for some uniform constant $c>0$
(in fact we will not differentiate between positive uniform constants and denote them all by $c$)
\begin{equation}\label{w-intermediate-estimate}
\begin{split}
I_1 \leq &-\rho\lambda^2e^{\lambda u}(1+\rho e^{\lambda u})|\nabla u|_g^2\varphi v \\
&- \Bigl( 1- 2\left|\frac{f'(u)}{f(u)}\right| \varepsilon \Bigl) \|\ssff\|^2\varphi v-3\left(\frac{f'(u)}{f(u)}\right)^2|\nabla u|_g^2 \varphi v \\
&+2c \left|\frac{f'(u)}{f(u)}\right| \varepsilon^{-1} \varphi v^3 -2 \rho\lambda e^{\lambda u} g(\nabla u,\nabla v)\varphi.
\end{split}
\end{equation}
We now want to estimate the last term above. 
By Proposition \ref{g(du,dv)} we have for some uniform constant $c>0$
\begin{align*}
-2 \rho\lambda e^{\lambda u} g(\nabla u,\nabla v)\varphi
& \le 2\rho\lambda e^{\lambda u}|g(\nabla u,\nabla v)|\varphi \\
& \le 2\rho\lambda e^{\lambda u} \Bigl(\|h\| |\nabla u |^2_g + c |\nabla u |^2_g v \Bigr)\varphi.
\end{align*}
We estimate this further for any $\varepsilon'>0$ and using \eqref{UpperBoundNormGradu} in the last step
\begin{align*}
-2 \rho\lambda e^{\lambda u} g(\nabla u,\nabla v)\varphi 
& \le \frac{2\rho\lambda e^{\lambda u} |\nabla u |^2_g}{\sqrt{2(1-\varepsilon') v} } 
\sqrt{2(1-\varepsilon') v} \|\ssff\| \varphi + 2c \rho\lambda e^{\lambda u} |\nabla u |^2_g \varphi v \\
& \le \frac{\rho^2\lambda^2 e^{2\lambda u} |\nabla u |^4_g}{(1-\varepsilon') v} \varphi +
(1-\varepsilon') \|\ssff\|^2 \varphi v + 2c \rho\lambda e^{\lambda u} |\nabla u |^2_g \varphi v \\
& \le \frac{\rho^2\lambda^2 e^{2\lambda u}}{(1-\varepsilon')} |\nabla u |^2_g \varphi v +
(1-\varepsilon') \|\ssff\|^2 \varphi v + 2c \rho\lambda e^{\lambda u} |\nabla u |^2_g \varphi v.
\end{align*}
Choosing, for any given $\varepsilon' \in (0,1)$, an $\varepsilon > 0$ sufficiently small such that  
$\varepsilon' > 2\left|\frac{f'(u)}{f(u)}\right| \varepsilon$ and plugging the
last estimate into \eqref{w-intermediate-estimate}, we arrive at
\begin{equation}\label{ChoiceOfLambda}
\begin{split}
I_1 \leq \ &-\rho\lambda e^{\lambda u} \Bigl( \lambda - 2c\Bigr) |\nabla u|_g^2\varphi v 
-3\left(\frac{f'(u)}{f(u)}\right)^2|\nabla u|_g^2 \varphi v \\
&+ 2c \left|\frac{f'(u)}{f(u)}\right| \varepsilon^{-1}  \varphi v^3 + \frac{\varepsilon'}{(1-\varepsilon')} \rho^2\lambda^2 e^{2\lambda u} |\nabla u |^2_g \varphi v.
\end{split}
\end{equation}
Set $\varepsilon'= e^{-\lambda u}$ and $\rho=1/2$. Choose $\overline{\lambda}>0$ 
so that for every $\lambda>\overline{\lambda}$ 
$$
\frac{\rho}{1-e^{-\lambda u}}\le \frac{3}{4}.
$$
Then we can estimate $I_1$ even further by (recall $|\nabla u|_g \leq v$ by Proposition \ref{1234})
\begin{equation*}
\begin{split}
I_1 \leq \ &- \frac{1}{8}\lambda e^{\lambda u} \Bigl( \lambda - c \Bigr) |\nabla u|_g^2\varphi v 
+c \lambda e^{\lambda u} \varphi v^3. \end{split}
\end{equation*}
We want to point out that the above estimates follows by considering $\overline{\lambda}$ to be large enough so that the second and third term in \eqref{ChoiceOfLambda} can be estimated by the second term in the equation above.

\noindent
Summarizing, we arrive at the following intermediate estimate
\begin{align}\label{almostfinal}
(\partial_s+\Delta)w \leq - \frac{1}{8}\lambda e^{\lambda u}  \Bigl( \Bigl( \lambda - c \Bigr) |\nabla u|_g^2 -cv^2 \Bigr) w.
\end{align}
We want to turn this into a differential inequality for the supremum 
$$
v_{\sup}(s) = \sup_{p\in M} v(p,s).
$$
Let us assume that there exists some $s_0 \in [0,T]$ such that $v_{\sup}(s_0)>2$, otherwise the statement is trivial. 
Since by Proposition \ref{envelope},  $v_{\sup}(s)$ is locally Lipschitz,
$v_{\sup}(s) > 2$ in an open intervall $I=(a,b) \subset [0,T]$ containing $s_0$. We take the minimal possible such $a\geq 0$, such that
by continuity of $v_{\sup}(s)$ we have either $a=0$ or $v_{\sup}(a) = 2$. \medskip

Consider $s \in (a,b)$ and 
a sequence $(p_k(s)) \subset M$ satisfying \eqref{omori-sup}. Then for $k\in \N$ sufficiently large, 
$v(p_k(s),s) > 2$ and we establish a differential evolution inequality for $v$ at those points as follows.
We consider $v$ and $w$ evaluated at $(p_k(s),s)$ without making it notationally explicit.
Since $v \geq 2$, we have $-4 \geq -v^2$ and from $(i)$ in Proposition \ref{1234} we find
\begin{equation}\label{du-lower-estimate}
|\nabla u|_g^2=v^2-1\ge v^2 - \frac{v^2}{4}= \frac{3}{4}v^2.
\end{equation}
Choosing $\lambda>\overline{\lambda}$ sufficiently large (note that these choices do not depend on $u$)
the right hand side of \eqref{almostfinal}, evaluated at $(p_k(s))$ for $k \in \N$ sufficiently large, turns negative and we conclude
$$
(\partial_s+\Delta)w(p_k(s),s) \le 0.
$$
This implies by Proposition \ref{envelope} for any $s \in (a,b)$ in the limit $k \to \infty$
$$
\partial_s w_{\sup}(s) \le 0.
$$
Thus for any $s \in (a,b)$ we conclude $w(\cdot,s)\le w_{\sup}(s) \le w_{\sup}(a)$. 
In particular, we find for any $(p,s) \in M \times (a,b)$ and some constant $c>0$, depending
only on $\mathcal{H}$, $u(s=a)$ and the ambient geometry, that (note that $e^{\rho e^{\lambda u}} > 1$)
\begin{align}\label{vu}
v(p,s)\le \exp \Bigl( \rho e^{\lambda u_{\sup}(a)} \Bigr) v_{\sup}(a) < c v_{\sup}(a),
\end{align}
where the second estimate holds, provided $u$ is bounded uniformly from above. 
Now, since we have either $a=0$ or $v_{\sup}(a) = 2$, we conclude that $v$ is uniformly bounded. 

\begin{cor}\label{c1-corr}
We continue in the Setting \ref{setting}. Assume that $u$ is uniformly bounded from above. 
Then $v$ is uniformly bounded and hence, provided $f$ is uniformly bounded, as assumed in Setting \ref{setting}, 
$|\NT u \, |_{\widetilde{g}}$ is uniformly bounded as well.
\end{cor}

\section{$C^2$-estimates: Bounds of the second fundamental form}\label{C2EstimatesSec}

Uniform boundedness of $\| h \|$ and hence also of the mean curvature $H$ has been established already in
Proposition \ref{lower-bound-u-1}. Now, as computed in the preceeding work by the first named author \cite[(2.15)]{mio}
\begin{equation}
\ssff_{ij}=-\frac{f(u)}{\sqrt{f(u)^2-|\widetilde{\nabla} u|_{\widetilde{g}}^2}}\left(u_{ij}-\widetilde{\Gamma}^k_{ij}u_k-2\frac{f(u)f'(u)}{f(u)^2}u_iu_j+f(u)f'(u)\widetilde{g}_{ij}\right).
\end{equation}
From here it is clear in view of uniform bounds of $f$ and its derivatives, as well as Corollary \ref{c1-corr} that
each $u_{ij}$ is uniformly bounded. We thus arrive at the $C^2$ estimates

\begin{prop}\label{c2-prop}
We continue in the Setting \ref{setting}. Assume that $u$ is uniformly bounded from above. 
Then $\| h \|$ and hence also of the mean curvature $H$ are uniformly bounded and hence
$|\NT^2 u \, |_{\widetilde{g}}$ is uniformly bounded as well.
\end{prop}

Taken altogether, results on the last three sections yield the following

\begin{thm}\label{C012}
Consider Setting \ref{setting} and a solution $u \in C^{4,\alpha}(M\times[0,T])$ to \eqref{PGMCF}.

\begin{enumerate}
\item Impose Assumptions \ref{assumptions} (1) and (2). 
Then $u, |\NT u \, |_{\widetilde{g}}$ and $|\NT^2 u \, |_{\widetilde{g}}$ are bounded uniformly for finite $T>0$,
with bounds possibly depending on $T$.

\item Impose Assumptions \ref{assumptions} (1)-(3). 
Then $u, |\NT u \, |_{\widetilde{g}}$ and $|\NT^2 u \, |_{\widetilde{g}}$ are bounded uniformly independent of $T>0$.
Moreover, $\| \partial_s u \|_\infty$ is exponentially decreasing. 
\end{enumerate} 
\end{thm}

\section{Long time existence and convergence}\label{conv-section}

As before, we continue in the Setting \ref{setting} and consider a local solution $u$ to \eqref{PGMCF}
extended to a maximal time intervall $u \in C^{4,\alpha}(M\times[0,T_{\max}))$. Let us 
assume without loss of generality that $T_{\max}>0$ is finite. The H\"older norm is
bounded for each compact intervall in $[0,T_{\max})$, but may a priori blow up the closer we get to $T_{\max}$.
The main point of this section is show that a posteriori this does not happen. \medskip

We first use uniform estimates from Theorem \ref{C012} to establish uniform ellipticity 
in the sense of \eqref{uniform-elliptic} for the Laplacian $\Delta$ of $g=g(s) = F(s)^*\gbar$.
Recall also the constant $\Lambda>0$ in the definition of uniform ellipticity in \eqref{uniform-elliptic}.

\begin{prop}\label{uni-ell-thm}
Consider a solution $u\in C^{4,\alpha}(M\times[0,T_{\max}))$ to \eqref{PGMCF}. 
\begin{enumerate}
\item If $u, |\NT u \, |_{\widetilde{g}}$ and $|\NT^2 u \, |_{\widetilde{g}}$ are bounded uniformly for any finite $T_{\max}>0$,
then $\Delta$ is uniformly elliptic for each $s \in [0,T_{\max})$ with $\Lambda>0$ bounded for any finite 
maximal time $T_{\max}$.\medskip

\item If $u, |\NT u \, |_{\widetilde{g}}$ and $|\NT^2 u \, |_{\widetilde{g}}$ are bounded uniformly independent of $T_{\max}>0$,
then $\Delta$ is uniformly elliptic for each $s \in [0,T_{\max})$ where $\Lambda>0$ can be chosen independent of $T_{\max}$.
\end{enumerate}
\end{prop}

\begin{proof}
From \eqref{laplacian-g} we obtain after cancellations
\begin{equation}\label{laplacian-g2}
\begin{split}
\Delta u&=\frac{1}{f(u)^2-|\NT u \, |_{\widetilde{g}}^2}
\Bigl( \widetilde{\Delta} + \widehat{\Delta} \Bigr) u\\
&+\frac{|\NT u\, |_{\widetilde{g}}^2}{(f(u)^2-|\NT u\, |_{\widetilde{g}}^2)} \left(
\frac{f(u)f'(u)}{f(u)^2-|\NT u\, |_{\widetilde{g}}^2} -(m-1)\frac{f'(u)}{f(u)}\right).
\end{split}
\end{equation}
Thus, in view of uniform bounds, it suffices to prove uniform ellipticity for $(\widetilde{\Delta} + \widehat{\Delta})$.
We compute from \eqref{deltahat} in local coordinates
\begin{equation*}
\begin{split}
\widetilde{\Delta} + \widehat{\Delta} = \frac{1}{f(u)^2-|\NT u\, |_{\widetilde{g}}^2}
\left( -\widetilde{g}^{ij} - \frac{\widetilde{g}^{iq} u_q \widetilde{g}^{jm}u_m}{f(u)^2-|\NT u\, |_{\widetilde{g}}^2}\right) 
\left(u_{ij}-\widetilde{\Gamma}_{ij}^ku_k\right)
\end{split}
\end{equation*}
From here we obtain for the symbol of $(\widetilde{\Delta} + \widehat{\Delta})$ in local coordinates 
\begin{equation}
\begin{split}
\sigma( \widetilde{\Delta} + \widehat{\Delta}) (p,\xi) &=
 \frac{1}{f(u)^2-|\NT u\, |_{\widetilde{g}}^2}
\left( \widetilde{g}^{ij} \xi_i \xi_j + \frac{\widetilde{g}^{iq} u_q \xi_i \widetilde{g}^{jm}u_m \xi_j}{f(u)^2-|\NT u\, |_{\widetilde{g}}^2}\right) 
\\ &= \frac{1}{f(u)^2-|\NT u\, |_{\widetilde{g}}^2}
\left( \| \xi \|^2_{\widetilde{g}} + \frac{\widetilde{g}(\di u,  \xi)^2}{f(u)^2-|\NT u\, |_{\widetilde{g}}^2}\right) 
\end{split}
\end{equation}
This implies uniform ellipticity as asserted.
\end{proof}

We can now establish H\"older regularity of the gradient function.

\begin{prop}\label{uv-hoelder} \ \\[-2mm]

\noindent Consider a solution $u \in C^{4,\alpha}(M\times[0,T_{\max}))$ to \eqref{PGMCF} and assume that $\mathcal{H}$ is bounded.

\begin{enumerate}
\item If $u, |\NT u \, |_{\widetilde{g}}$ and $|\NT^2 u \, |_{\widetilde{g}}$ are bounded uniformly for finite $T_{\max}>0$,
then $$u, v \in C^{\alpha}(M\times[0,T_{\max}]),$$ with the H\"older norm bounded for finite $T_{\max}>0$. \medskip

\item If $u, |\NT u \, |_{\widetilde{g}}$ and $|\NT^2 u \, |_{\widetilde{g}}$ are bounded uniformly independent of $T_{\max}>0$,
then $$u, v \in C^{\alpha}(M\times[0,T_{\max}]),$$ with a $T_{\max}$-independent bound for the H\"older norm. 
\end{enumerate}
\end{prop}

\begin{proof}
Consider the evolution equation for the gradient function $v$, as derived in Theorem \ref{EvolutionOfvTHM}.
Since the right hand side of \eqref{EvolutionOfv} is bounded uniformly (with bounds possibly depending 
on $T$ depending on whether $u, |\NT u \, |_{\widetilde{g}}$ and $|\NT^2 u \, |_{\widetilde{g}}$ are bounded 
independent of $T$ or not), H\"older regularity follows by uniform ellipticity of $\Delta$ in Proposition
\ref{uni-ell-thm} and the Krylov-Safonov estimates in the first statement of Proposition \ref{krylov-safonov-lemma}.
The statement for $u$ follows in exactly the same way from the evolution equation \eqref{PGMCF}.
\end{proof}

We can now bootstrap to improve upon regularity of $u$. 

\begin{prop}\label{v-hoelder}
Consider a solution $u \in C^{4,\alpha}(M\times[0,T_{\max}))$ to \eqref{PGMCF}. Assume that
$\mathcal{H} \in C^{\ell,\alpha}(M)$ for some $\ell\in\mathbb{N}_0$. Then the following is true.

\begin{enumerate}
\item If $u, |\NT u \, |_{\widetilde{g}}$ and $|\NT^2 u \, |_{\widetilde{g}}$ are bounded uniformly for finite $T_{\max}>0$,
then $$u \in C^{2+\ell,\alpha}(M\times[0,T_{\max}]),$$ with the H\"older norm bounded for finite $T_{\max}>0$. \medskip

\item If $u, |\NT u \, |_{\widetilde{g}}$ and $|\NT^2 u \, |_{\widetilde{g}}$ are bounded uniformly independent of $T_{\max}>0$,
then $$u\in C^{2+\ell,\alpha}(M\times[0,T_{\max}]),$$ with a $T_{\max}$-independent bound for the H\"older norm. 
\end{enumerate}
\end{prop}

\begin{proof}
Consider the evolution equation \eqref{PGMCF} for the solution $u$. By 
Proposition \ref{uv-hoelder}, the right hand side of \eqref{PGMCF} as well as 
the coefficients of $\Delta$ (cf. \eqref{laplacian-g2}) lie in $C^{\alpha}(M\times[0,T_{\max}])$.
Thus, by Proposition \ref{krylov-safonov-lemma} (ii), we conclude 
$$u \in C^{2,\alpha}(M\times[0,T_{\max}]).$$
Now we can bootstrap exactly as at the end of the proof of Theorem \ref{short-time-theorem}.
\end{proof}

Therefore, assuming that $\mathcal{H} \in C^{2,\alpha}(M)$, we have $u \in C^{4,\alpha}(M\times[0,T_{\max}])$
and hence by Theorem \ref{short-time-theorem} we can restart the flow with $u(T_{\max}) \in C^{4,\alpha}(M)$ as a new initial condition. We arrive 
at the following main result. 

\begin{thm}\label{long-time-result}
Consider the Setting \ref{setting}. Then the mean curvature flow \eqref{PGMCF}.

\begin{enumerate}
\item admits a global solution $u \in C^{\ell + 2,\alpha}(M\times[0,\infty))$ with in $[0,\infty)$ locally uniformly bounded
H\"older norm, if Assumptions \ref{assumptions} (1) and (2) hold;

\item admits a global solution $u \in C^{\ell + 2,\alpha}(M\times[0,\infty))$ with uniformly bounded H\"older norm, if 
Assumptions \ref{assumptions} (1) - (3) hold. Moreover, $\| \partial_s u \|_\infty$ is exponentially decreasing. 
\end{enumerate} 
\end{thm}

\noindent This proves Theorem \ref{theorem-main-long}. \medskip

It remains to discuss convergence under the conditions of Theorem \ref{long-time-result} (ii).
First we note that exponential decay of $\| \partial_s u \|_\infty$ implies that $u(s)$
admits a well-defined limit $u^* \in L^\infty(M)$ as $s \to \infty$. We need to conclude 
at least that $u^* \in C^2(M)$ in order for $u^*$ to admit a well-defined mean curvature $H^*$,  
which can then be shown to equal $\mathcal{H}$. This is then our final main result.

\begin{thm}\label{convergence-theorem}
Consider the Setting \ref{setting} and impose Assumptions \ref{assumptions} (1) - (3). Assume that $M$ is 
the open interior of a compact manifold $\overline{M}$ with boundary $\partial M$.
Then the prescribed mean curvature flow \eqref{PGMCF}, starting at $u_0$ exists for 
all times and converges to $u^* \in L^\infty(M)$ as $s \to \infty$. Moreover, $u^* \in C^{\ell+2}$ in
the open interior $M$ with well-defined mean curvature $H^*\equiv \mathcal{H}$.
\end{thm}

\begin{proof}
As mentioned above, convergence to $u^* \in L^\infty(M)$ follows from the exponential decay of $\| \partial_s u \|_\infty$; therefore it
remains to prove that the limit is twice differentiable in $M$. 
Let $x: \overline{M} \to \R^+$ be a defining function of $\partial M$. Then, cf.  
\cite[Proposition 11.2]{bruno}, for any $\varepsilon > 0$ the inclusion of weighted H\"older spaces 
$$
\iota: C^{\ell + 2,\alpha}(M) \hookrightarrow x^{-\varepsilon} C^{\ell + 2,\alpha}(M),
$$
is compact. Consider the global solution $u \in C^{\ell + 2,\alpha}(M \times [0,\infty))$,
whose existence follows by the previous Theorem \ref{long-time-result}. Since the sequence $(u(s))_s \subset C^{\ell + 2,\alpha}(M)$
is uniformly bounded, by compactness of $\iota$, there exists a convergent subsequence $(u(s_n))_n\subset x^{-\varepsilon}C^{\ell + 2,\alpha}(M)$.
Consequently the pointwise limit $u^*$ lies in $x^{-\varepsilon}C^{\ell + 2,\alpha}(M)$. In particular it admits a well-defined 
mean curvature $H^*$. By \eqref{H-exp}, $H(s_n) - \mathcal{H}$ converges to zero and hence indeed $H^* = \mathcal{H}$.
\end{proof}
 
\noindent This proves Theorem \ref{theorem-main}.
 
\bibliographystyle{amsalpha-lmp}

\end{document}